\documentclass[11pt,letterpaper]{amsart}

\usepackage[T1]{fontenc}
\usepackage{lmodern}
\usepackage[expansion=false]{microtype}
\usepackage{amsmath,amssymb,mathtools,mathrsfs}
\usepackage[margin=1.08in]{geometry}
\usepackage{xcolor}
\usepackage[
  colorlinks=true,
  linkcolor=blue,
  citecolor=red,
  urlcolor=blue
]{hyperref}

\usepackage{amsrefs}

\numberwithin{equation}{section}
\theoremstyle{plain}
\newtheorem{theorem}{Theorem}[section]
\newtheorem{maintheorem}{Theorem}

\newtheorem{proposition}[theorem]{Proposition}
\newtheorem{lemma}[theorem]{Lemma}
\newtheorem{corollary}[theorem]{Corollary}
\theoremstyle{definition}

\newtheorem{example}[theorem]{Example}
\newtheorem{question}[theorem]{Question}
\newtheorem*{bedfordquestion}{Question 2}
\theoremstyle{remark}
\newtheorem{remark}[theorem]{Remark}

\newcommand{\C}{\mathbb C}
\newcommand{\D}{\mathbb D}
\newcommand{\PSH}{\operatorname{PSH}}
\newcommand{\MA}{\operatorname{MA}}
\newcommand{\cM}{\mathcal M}
\newcommand{\ddc}{dd^c}
\newcommand{\loc}{\mathrm{loc}}
\newcommand{\supp}{\operatorname{supp}}
\newcommand{\Clust}{\operatorname{Clust}}

\hypersetup{
  pdftitle={Prescribed Limits and Cluster Sets of Complex Monge-Ampere Measures},
  pdfauthor={Xiangsen Qin},
  pdfsubject={Pluripotential theory and the complex Monge-Ampere operator},
  pdfkeywords={plurisubharmonic functions, complex Monge-Ampere operator, prescribed limits, cluster sets}
}

\begin{document}

\title[Prescribed limits and cluster sets]
{Prescribed Limits and Cluster Sets of Complex
Monge--Amp\`ere Measures}

\author[X. Qin]{Xiangsen Qin}
\address{Xiangsen Qin: Chern Institute of Mathematics and LPMC, Nankai University \\
Tianjin 300071, China}
\email{qinxiangsen@nankai.edu.cn}

\begin{abstract}
Motivated by Bedford's question, we study the possible cluster sets,
in the vague topology, of Monge--Amp\`ere measures associated with
bounded plurisubharmonic sequences $u_j\to\varphi$ locally in $L^1$,
without assuming a common bound.
On a bounded domain $\Omega\subset\mathbb C^n$ ($n\geq2$), and for
a bounded maximal plurisubharmonic $\varphi$, the possible sets of
all vague subsequential measure limits are precisely the vaguely
closed sets of nonnegative Radon measures, including the empty set.
This conclusion holds for every bounded plurisubharmonic $\varphi$
when $\Omega$ is hyperconvex, pseudoconvex and Runge, or a
domain obtained by removing a relatively closed pluripolar set from
a bounded hyperconvex domain.
Under these hypotheses, every nonnegative Radon measure $\nu$ can
also be realized by such a sequence with $(dd^cu_j)^n\to\nu$ vaguely.
The main analytic result holds on every bounded domain: for any
bounded plurisubharmonic $\varphi$ and nonnegative Radon measure $\mu$,
we construct $u_j\leq\varphi$ converging to $\varphi$ in
$L^p(\Omega)$ for every $1\leq p<\infty$, with
$(dd^cu_j)^n\to(dd^c\varphi)^n+\mu$ vaguely.
The approximants can be chosen smooth when $\varphi$ is continuous. On arbitrary pseudoconvex domains, prescribed measure limits can
also be realized using locally bounded approximants. Under an
additional localization condition, the approximants can be chosen
globally bounded when the limit function is bounded.
\end{abstract}


\maketitle
\tableofcontents
\section{Introduction}\label{sec:introduction}

The Bedford--Taylor theory gives continuity of Monge--Amp\`ere
measures along decreasing sequences of locally bounded
plurisubharmonic (psh) functions with a locally bounded limit
\cite{BedfordTaylor1982}.
Under local $L^1$ convergence alone, however, the measures may have
subsequential limits different from the Monge--Amp\`ere measure of
the limiting function.
Bedford's question, stated below, asks what can be said about the
set of these subsequential limits when each approximating function
is bounded, but no common bound is assumed.

We write $\PSH(\Omega)$ for the psh functions on a
domain $\Omega\subset\C^n$ and use the normalization
\[
 d^c=\frac{\partial-\bar\partial}{4\pi i},\qquad
 dd^c=\frac{i}{2\pi}\partial\bar\partial.
\]
For locally bounded psh $u$, let $\MA(u)=(dd^cu)^n$ denote its Monge--Amp\`ere measure.

\begin{bedfordquestion}[Bedford, as recorded in {\cite{DGZ}*{Question~2}}]
Let $\Omega\subset\C^n$ be a domain and let
$\varphi\in\PSH(\Omega)\cap L^\infty(\Omega)$.  Suppose that
\[
 \varphi_j\in\PSH(\Omega)\cap L^\infty(\Omega),\qquad
 \varphi_j\longrightarrow\varphi\quad\text{in }L^1_{\loc}(\Omega).
\]
What can be said about the set of cluster points of $(dd^c\varphi_j)^n$?
\end{bedfordquestion}

In dimension one, every such sequence satisfies
$dd^c\varphi_j\to dd^c\varphi$ vaguely, so its cluster set is
$\{dd^c\varphi\}$; see Proposition~\ref{prop:dimension-one}.
The results below concern higher dimensions ($n\geq2$).

In this paper we formulate the question in the vague
topology; the original question does not specify a measure topology.
Write $\cM_+(\Omega)$ for the nonnegative Radon measures, including zero.
These measures are finite on compact subsets and may have infinite total
mass.  Vague convergence means convergence against every test function in
$C_c(\Omega)$.  For a sequence $(\nu_j)$ in $\cM_+(\Omega)$, its full
vague cluster set is
\[
 \Clust_{\mathrm v}(\nu_j)
 =\{\nu\in\cM_+(\Omega):\nu_{j_\ell}\to\nu\text{ vaguely
 for some }j_1<j_2<\cdots\}.
\]
A set $F\subset\cM_+(\Omega)$ is called vaguely closed if
every sequence in $F$ that converges vaguely to a measure in
$\cM_+(\Omega)$ has its limit in $F$.

For a fixed limit function $\varphi$, we distinguish individual
measure limits obtained from all admissible approximating sequences
from full cluster sets obtained by choosing one such sequence.  We
classify these possibilities under the hypotheses below; the cluster
set of a sequence prescribed in advance depends on that sequence.  Question~1 in \cite{DGZ}, which
concerns decreasing approximation to a possibly unbounded psh function,
has different approximation requirements.

The first theorem classifies the possible full cluster sets in four
cases.  Recall that a locally bounded psh function $\varphi$ is maximal
if and only if $\MA(\varphi)=0$.

\begin{maintheorem}
\label{thm:classification}\label{cor:maximal-background}\label{cor:question2-cases}
Let $\Omega\subset\C^n$ ($n\geq2$) be a bounded domain and let
$\varphi\in\PSH(\Omega)\cap L^\infty(\Omega)$.  Suppose that one of
the following holds:
\begin{enumerate}
 \item $\varphi$ is maximal on $\Omega$;
 \item $\Omega$ is hyperconvex;
 \item $\Omega$ is pseudoconvex and Runge in $\C^n$;
 \item $\Omega=G\setminus E$, where $G\subset\C^n$ is
 a bounded hyperconvex domain and $E$ is relatively closed and pluripolar
 in $G$.
\end{enumerate}
Then for every $F\subset\cM_+(\Omega)$, the following are equivalent:
\begin{enumerate}
 \item[$\mathrm{(i)}$] $F$ is vaguely closed;
 \item[$\mathrm{(ii)}$] there exist
 $u_j\in\PSH(\Omega)\cap L^\infty(\Omega)$ such that
 \[
 u_j\to\varphi\quad\text{in }L^1_{\loc}(\Omega),\qquad
 \Clust_{\mathrm v}(\MA(u_j))=F.
 \]
\end{enumerate}
The empty set is included.  In addition, every
$\nu\in\cM_+(\Omega)$ is the vague limit of the entire measure
sequence $\MA(u_j)$ for some such approximating sequence.
In cases \textup{(2)--(4)}, the approximants in both assertions can be
chosen smooth.  In case \textup{(1)}, they can be chosen smooth if
$\varphi$ is continuous.
\end{maintheorem}

On bounded open subsets of $\C^2$, Cegrell \cite{Cegrell1984}
studied the measure limits obtained by approximating a fixed locally
bounded psh function with a uniformly locally bounded psh sequence,
using distributional convergence for both potentials and measures.
He proved that this set is star-shaped about the Monge--Amp\`ere
measure of the limit function and that its elements give no mass
to pluripolar sets
\cite{Cegrell1984}*{Propositions~1(a) and~2(a)}.
For nonnegative Radon measures, distributional convergence is
equivalent to vague convergence.

Here each approximant is globally bounded, but no common local bound
is imposed on the sequence.
Under the hypotheses of Theorem~\ref{thm:classification}, every
nonnegative Radon measure, including a point mass, is the vague limit
of an entire sequence of Monge--Amp\`ere measures.
Moreover, for each fixed limit function covered by that theorem,
the possible full cluster sets are exactly the vaguely closed
subsets of $\cM_+(\Omega)$, including the empty set.

Cases \textup{(2)--(4)} apply to every bounded psh function on the
indicated domain; case \textup{(1)} requires no geometric condition on
the bounded domain.  By Chen's theorem
\cite{Chen2021}*{Theorem~1.1}, bounded pseudoconvex domains with H\"older boundary
are hyperconvex and hence satisfy \textup{(2)}.  Punctured balls, punctured polydiscs,
and the Hartogs triangle are further examples.  The compact-hull
condition in Section~\ref{sec:geometry} gives a broader sufficient
condition.

Even when smooth bounded psh functions converge to zero in
$L^p(\Omega)$ for every $1\leq p<\infty$, their Monge--Amp\`ere
measures may converge vaguely to a nonzero measure---in particular,
a Dirac measure.
Indeed, on any bounded domain $\Omega\subset\C^n$ ($n\geq2$), and
for any $a\in\Omega$, the following Theorem~\ref{thm:fixed-addition} with
$\varphi=0$ gives smooth bounded psh functions $u_j$ satisfying
\[
 u_j\leq0,\qquad
 \|u_j\|_{L^p(\Omega)}\to0\quad(1\leq p<\infty),
 \qquad \MA(u_j)\to\delta_a\quad\text{vaguely},
\]
where $\delta_a$ is the Dirac measure at $a$.
Such concentration would be impossible under a common local bound,
which would force $\MA(u_j)\to0$ by
Proposition~\ref{prop:common-bound-zero}.
Since $u_j\leq0$, the sequence must therefore lack a common lower
bound on some compact subset of $\Omega$.

The following theorem supplies the analytic construction used in
Theorem~\ref{thm:classification}. It applies to every bounded psh
function on a bounded domain, independently of the geometric
hypotheses in that classification. Here and throughout, $L^p$ norms
are taken with respect to Lebesgue measure.

\begin{maintheorem}\label{thm:fixed-addition}
Let $\Omega\subset\C^n$ ($n\geq2$) be a bounded domain, let
$\varphi\in\PSH(\Omega)\cap L^\infty(\Omega)$, and let
$\mu\in\cM_+(\Omega)$.  Then there are bounded psh functions $u_j$ such that
\[
 u_j\leq\varphi\quad\text{on }\Omega,\qquad
 \|u_j-\varphi\|_{L^p(\Omega)}\longrightarrow0
 \quad(1\leq p<\infty),
\]
and
\[
 \MA(u_j)\longrightarrow\MA(\varphi)+\mu\quad\text{vaguely}.
\]
The same sequence works for every finite $p\geq 1$.  If $\varphi$ is
continuous, the functions $u_j$ can be chosen smooth.
\end{maintheorem}

The truncation formula of Andersson, Witt Nystr\"om, and Wulcan
\cite{AWNW}*{Theorem~1.5}, whose analytic-singularity case follows
from B\l ocki's \cite{Blocki2019}*{Theorem~1}, provides the starting
point; see \cite{AWNW}*{Remark~5.7}.
Their Example~5.8 illustrates how truncations approaching a fixed
singular potential can produce arbitrarily large limiting mass on
a divisor.
Our construction instead uses these formulas to prescribe the
measure limit of perturbations tending to zero.
The main step is to retain the prescribed top-degree measure while
making every lower-degree current small in local trace mass.
We achieve this by combining approximation on parallel slices with
the scales $\varepsilon$ and $\varepsilon^{-1/(n-1)}$.
Chern--Levine--Nirenberg estimates then make the mixed terms vanish
when the perturbations are added to any bounded psh function
$\varphi$, without requiring continuity.
The parameter choices also ensure global boundedness,
approximation from below, and convergence in $L^p(\Omega)$ for every
$1\leq p<\infty$ along the same sequence.

Theorem~\ref{thm:tracking} extends the construction to sequences of
bounded psh functions $b_k$ and prescribed measures $\mu_k$,
retaining $u_k\leq b_k$ and convergence of $u_k-b_k$ to zero in
$L^p(\Omega)$ for every $1\leq p<\infty$.
This extension is used below to establish the structure of reachable
limits and to realize full cluster sets.

For a fixed bounded psh function $\varphi$, define the reachable set
\[
 \mathscr R_b(\Omega,\varphi)
 :=\left\{\nu\in\cM_+(\Omega):\begin{array}{l}
 \text{there exist }u_j\in\PSH(\Omega)\cap L^\infty(\Omega),\\
 u_j\to\varphi\text{ in }L^1_{\loc}(\Omega),\quad
 \MA(u_j)\to\nu\text{ vaguely}
 \end{array}\right\}.
\]
Requiring every $u_j$ to be smooth defines
$\mathscr R_{\mathrm{sm}}(\Omega,\varphi)$.
Equivalently, one may require every $u_j$ to be continuous,
by Lemma~\ref{lem:continuous-approximation}. These sets collect
individual measure limits across all approximating sequences.  For
sets of measures, $\mathcal A+\mathcal B$ denotes
$\{\alpha+\beta:\alpha\in\mathcal A,\ \beta\in\mathcal B\}$.

\begin{maintheorem}
\label{thm:structure}\label{thm:addition}\label{thm:zero-intrinsic}\label{thm:clusters}
Let $\Omega\subset\C^n$ ($n\geq2$) be a bounded domain and let
$\varphi\in\PSH(\Omega)\cap L^\infty(\Omega)$.
For $a\in\{b,\mathrm{sm}\}$, the set
$\mathscr R_a(\Omega,\varphi)$ is vaguely closed and satisfies
\[
 \mathscr R_a(\Omega,\varphi)+\cM_+(\Omega)
 =\mathscr R_a(\Omega,\varphi),\qquad
 0\in\mathscr R_a(\Omega,\varphi)
 \ \Longleftrightarrow\
 \mathscr R_a(\Omega,\varphi)=\cM_+(\Omega).
\]
A set $F\subset\cM_+(\Omega)$ is the full vague cluster set of
$\MA(u_j)$ for a bounded psh sequence $u_j\to\varphi$ in $L^1_{\loc}(\Omega)$
if and only if
\[
 F\subset\mathscr R_b(\Omega,\varphi)
 \quad\text{and}\quad F\text{ is vaguely closed}.
\]
This includes $F=\varnothing$.
If $\mathscr R_{\mathrm{sm}}(\Omega,\varphi)\ne\varnothing$, the same
conclusion holds with smooth bounded approximants and
$\mathscr R_{\mathrm{sm}}$ in place of $\mathscr R_b$.
\end{maintheorem}

The sequence version of Theorem~\ref{thm:fixed-addition},
Theorem~\ref{thm:tracking}, supplies the addition property in
Theorem~\ref{thm:structure}; diagonal selection gives its closure
and cluster-set assertions. To deduce
Theorem~\ref{thm:classification}, it therefore suffices to make zero
reachable. The addition property then gives
$\mathscr R_b(\Omega,\varphi)=\cM_+(\Omega)$, and every vaguely
closed subset is realizable. The smooth conclusions follow in the
same way from $0\in\mathscr R_{\mathrm{sm}}(\Omega,\varphi)$.

For a bounded maximal $\varphi$, the constant sequence already
makes zero reachable.
For a general bounded psh function on a pseudoconvex domain in
$\C^n$ ($n\geq2$), Lelong's theorem \cite{Lelong}, as stated in
\cite{BacklundPersson}*{Theorem~2.1, p.~263}, gives continuous
maximal psh functions defined on the whole domain and converging
to $\varphi$ in $L^1_{\loc}$.
Their Monge--Amp\`ere measures vanish, but the functions need not
be globally bounded.
Under the compact-hull condition
\eqref{eq:compact-hulls} in Section~\ref{sec:geometry}, we modify
these functions by pasting to obtain bounded continuous psh
approximants that remain maximal near successive compact sets.
Their measures therefore tend vaguely to zero, and smoothing
gives zero reachability in the smooth class.
For the pluripolar-deletion case, we extend $\varphi$ to the
larger hyperconvex domain, construct approximants there with
measure limit zero, and restrict them to the original domain.
This also covers examples where the compact-hull condition fails.
Whether such bounded approximation with measure limit zero
exists for every bounded psh function on every bounded pseudoconvex
domain remains open.

We also consider arbitrary pseudoconvex domains in $\C^n$ ($n\geq2$).
For every locally bounded psh limit function, every nonnegative
Radon measure can be realized as a vague limit using locally bounded
psh approximants. Under condition~\textup{(V)}, defined in
Section~\ref{sec:localization}, globally bounded approximants are
available when the limit function is bounded. The cylinder example
in that section shows why global boundedness cannot be retained
on every unbounded pseudoconvex domain.

The remainder of this paper is organized as follows.
Sections~\ref{sec:slices} and~\ref{sec:addition-construction}
develop the analytic construction and establish
Theorem~\ref{thm:fixed-addition} and its sequence version.
Section~\ref{sec:tracking-results} proves
Theorem~\ref{thm:structure}.
Section~\ref{sec:geometry} establishes zero reachability under
the stated hypotheses, completes the proof of
Theorem~\ref{thm:classification}, and discusses examples.
Section~\ref{sec:localization} proves the extensions to general
domains and discusses the obstructions to global boundedness.
Section~\ref{sec:dimension-one} discusses the restrictions imposed
by dimension and common local bounds, and states the remaining
approximation question. Two appendices supplement the main argument.
Appendix~\ref{app:divisor} computes the truncation limit when a
locally bounded psh function is placed inside the maximum.
Its treatment of discontinuous functions uses
\cite{ABW2019}*{Theorem~2.1}; this calculation is not used in the proof of the main results,
where mixed terms are controlled by Chern--Levine--Nirenberg estimates.
Appendix~\ref{app:metric} explains the distinction between
comparison in the fixed vague metric and convergence of signed
measure differences against continuous compactly supported test
functions, and shows how common local mass bounds connect them.

\section{Small potentials with prescribed measures}\label{sec:slices}\label{sec:twisted}

Throughout this section, $\Omega\subset\C^n$ ($n\geq2$) is a bounded domain.  We construct continuous bounded psh functions that approach
zero while their Monge--Amp\`ere measures approach a prescribed measure
and their lower-degree currents have small local trace mass.
For finitely many test functions, the construction uses only the finite
mass on their compact supports.  We approximate this mass by atoms,
spread each atom on a distinct parallel slice, and use logarithmic
truncations to obtain the required psh functions.

With the normalization fixed in the Introduction, the Poincar\'e--Lelong
formula is
\begin{equation}\label{eq:PL-normalization}
 \ddc\log|f|^2=[\operatorname{div}f]
\end{equation}
for every nonzero holomorphic function $f$, where $\operatorname{div}(f)$ denotes the divisor of $f$. We write $\lambda_{2d}$ for
Lebesgue measure on $\C^d\simeq\mathbb R^{2d}$.

If $u_1,\ldots,u_p\in\PSH(\Omega)\cap L^\infty_{\loc}(\Omega)$ and
$1\leq p\leq n$, then
\[
 \ddc u_1\wedge\cdots\wedge\ddc u_p
\]
is the Bedford--Taylor product \cite{BedfordTaylor1982}.  It is a positive
closed current with locally finite mass and is local in the plurifine, hence
in the ordinary, topology.  We use $\MA(u)$ only for locally bounded psh functions.  Such a
function is \emph{maximal} on an open set $U$ if comparison holds on every
relatively compact subdomain of $U$.  This is equivalent to $\MA(u)=0$ on
$U$ \cite{BacklundPersson}*{Theorem~2.3}.

An effective divisor is \emph{reduced} if each irreducible component
has multiplicity one. Let $D=\bigcup_{r=1}^N D_r$ be a reduced divisor on $\Omega$ with pairwise disjoint smooth components, and let
$i_r:D_r\hookrightarrow\Omega$ denote the inclusions. Its integration
current is
\[
 [D]=\sum_{r=1}^N[D_r]=\sum_{r=1}^N(i_r)_*1,
\]
where $1$ denotes the current on $D_r$ defined by the constant
function $1$.
For locally bounded psh functions
$u_1,\ldots,u_p$, where $0\leq p\leq n-1$, define
\begin{equation}\label{eq:divisor-restriction-product}
 [D]\wedge\ddc u_1\wedge\cdots\wedge\ddc u_p
 :=\sum_{r=1}^N(i_r)_*
 \bigl(\ddc(u_1|_{D_r})\wedge\cdots
             \wedge\ddc(u_p|_{D_r})\bigr).
\end{equation}
This is a positive closed current with locally finite mass; in top degree it is a
positive Radon measure.

\subsection{Measures on parallel slices}
Write $m=n-1$ and $z=(z_1,z')\in\C\times\C^m$. For $a\in\C$, set
\[
 U_a=\{z'\in\C^m:(a,z')\in\Omega\},
 \qquad i_a:U_a\to\Omega,\quad i_a(z')=(a,z').
\]
The image
\[
 D_a:=i_a(U_a)=\Omega\cap\{z_1=a\}
\]
is called a \emph{slice} of $\Omega$: it is the intersection of
$\Omega$ with the complex hyperplane obtained by fixing the first
coordinate.
A \emph{slice measure} on $\Omega$ means a finite sum
$\sigma=\sum_{r=1}^N(i_{a_r})_*\eta_r$, where the $a_r$ are distinct
and $\eta_r$ is a nonnegative Radon measure on $U_{a_r}$.  Explicitly,
\[
 \int_\Omega g\,d\sigma
 =\sum_{r=1}^N\int_{U_{a_r}}g(a_r,z')\,d\eta_r(z')
 \qquad(g\in C_c(\Omega)).
\]
Each slice $D_{a_r}$ is closed in $\Omega$, so its inclusion is proper and the
push-forward is a Radon measure supported on that slice.
When a measure is defined on a larger subset of $\C^m$, its push-forward
is restricted to $\Omega$.

Since $\Omega$ is bounded, choose
$B_R=B_{\C^m}(0,R)$ so that the $z'$-projection of $\overline\Omega$ lies in
its interior, and choose
\begin{equation}\label{eq:normalizing-radius}
 R_1:=2\sup_{z\in\Omega}|z_1|+2.
\end{equation}

\begin{lemma}\label{lem:logarithmic-potential}
Let $m\geq1$, $R>0$, $b\in B_R\subset\C^m$, $c>0$, and $\delta>0$.
Set
\[
 H_{b,\delta}(z')=c^{1/m}\log\bigl(|z'-b|^2+\delta^2\bigr).
\]
This is a smooth strictly psh function on $\C^m$, and
\begin{equation}\label{eq:explicit-density}
 (\ddc H_{b,\delta})^m
 =c\frac{m!}{\pi^m}
 \frac{\delta^2}{(|z'-b|^2+\delta^2)^{m+1}}\,d\lambda_{2m}.
\end{equation}
The measure has total mass $c$ and tends vaguely to $c\delta_b$ as
$\delta\downarrow0$.  For every $\rho>0$ its mass outside $B(b,\rho)$ is
\begin{equation}\label{eq:explicit-tail}
 c\left[1-\left(\frac{\rho^2}{\rho^2+\delta^2}\right)^m\right]
 \leq\frac{cm\delta^2}{\rho^2}.
\end{equation}
For fixed $b,c,\delta$, a sufficiently large constant $C$ makes
$W=H_{b,\delta}+C$ positive with $\log W$ psh on a neighborhood of
$\overline{B_R}$.
\end{lemma}

\begin{proof}
Write $a=c^{1/m}$ and $s=|z'-b|^2$. For $z'\ne b$, the complex Hessian of $H$ has eigenvalue
$a/(s+\delta^2)$ on the complex orthogonal complement of $z'-b$,
with multiplicity $m-1$, and eigenvalue
$a\delta^2/(s+\delta^2)^2$ on the complex line spanned by $z'-b$.
At $z'=b$, all eigenvalues equal $a/\delta^2$.  Thus its determinant
is $c\delta^2/(s+\delta^2)^{m+1}$, which gives
\eqref{eq:explicit-density} with our normalization.  Polar integration gives
\[
 \int_{B(b,\rho)}(\ddc H_{b,\delta})^m
 =2cm\delta^2\int_0^\rho
       \frac{r^{2m-1}}{(r^2+\delta^2)^{m+1}}\,dr
 =c\left(\frac{\rho^2}{\rho^2+\delta^2}\right)^m.
\]
Letting $\rho\to\infty$ gives total mass $c$; subtraction and
$1-(1+x)^{-m}\leq mx$ give \eqref{eq:explicit-tail}.

For every bounded continuous function $g$ on $\C^m$,
\[
 \left|\int_{\C^m}g\,(dd^cH_{b,\delta})^m-cg(b)\right|
 \leq c\sup_{|z'-b|<\rho}|g(z')-g(b)|
 +2\|g\|_\infty\frac{cm\delta^2}{\rho^2}.
\]
By continuity of $g$ at $b$, the first term can be made arbitrarily
small by choosing $\rho>0$ sufficiently small. For fixed $\rho$,
the second term tends to zero as $\delta\downarrow0$.

On a slightly larger closed ball, strict positivity and smoothness give
$\lambda,M>0$ with
$\ddc H_{b,\delta}\geq\lambda\ddc|z'|^2$ and
$dH_{b,\delta}\wedge d^cH_{b,\delta}\leq M\ddc|z'|^2$.
Choose $C$ so that $W>M/\lambda+1$ there.  Then
\[
 \ddc\log W
 =W^{-2}\bigl(W\ddc H_{b,\delta}
       -dH_{b,\delta}\wedge d^cH_{b,\delta}\bigr)\geq0.
\]
Adding $C$ does not change \eqref{eq:explicit-density}.
\end{proof}

\begin{proposition}[Approximation by measures on parallel slices]\label{prop:slices}
Let $\Omega\subset\C^n$ ($n\geq2$) be a bounded domain, and choose
$B_R$ and $R_1$ as above.  The measures of the form
\begin{equation}\label{eq:slice-measure}
 \sigma=\sum_{r=1}^N
 \left((i_{a_r})_*(\ddc W_r)^m\right)\big|_\Omega,
\end{equation}
where the $a_r$ are distinct, $\{z_1=a_r\}\cap\Omega\neq\varnothing$, and
each $W_r$ is supplied by Lemma~\ref{lem:logarithmic-potential}, are vaguely
dense in $\cM_+(\Omega)$.

For every such $\sigma$, let
\[
 D:=\bigcup_{r=1}^N\bigl(\{z_1=a_r\}\cap\Omega\bigr),
 \qquad F(z):=\prod_{r=1}^N\frac{z_1-a_r}{R_1}.
\]
Then $D=\operatorname{div}_\Omega(F)$ is reduced and there is a smooth
psh function $V_0> 0$ on a neighborhood of $\overline\Omega$ such that
\begin{equation}\label{eq:programmed-slice}
 [D]\wedge(\ddc V_0)^m=\sigma\quad\text{on }\Omega.
\end{equation}
Moreover, $\log|F|^2\leq0$ on $\Omega$.
\end{proposition}

\begin{proof}
Fix $g_1,\ldots,g_s\in C_c(\Omega)$,
$\mu\in\cM_+(\Omega)$, and $\eta>0$.  Choose $L\subset\subset\Omega$ whose interior
contains $\bigcup_{j=1}^s\supp g_j$.  Since $\mu(L)<\infty$ and the finitely many functions $g_j$ are
continuous on the compact set $L$, there exist points
$(a_r,b_r)\in L$ and positive weights $c_r$ such that
\[
 \left|
 \int_L g_j\,d\mu-\sum_{r=1}^N c_r g_j(a_r,b_r)
 \right|<\frac{\eta}{3}
 \qquad\text{for every }j.
\]
Thus the positive atomic measure
$\sum_{r=1}^N c_r\delta_{(a_r,b_r)}$ approximates $\mu|_L$
against all the prescribed test functions.  Since $\Omega$ is open, each first coordinate $a_r$ can be perturbed
slightly while keeping $(a_r,b_r)$ in $\Omega$. By continuity of the
finitely many test functions, we can choose pairwise distinct
coordinates $\widetilde a_r$ such that
$(\widetilde a_r,b_r)\in\Omega$ and
\[
 \sum_{r=1}^N c_r
 \bigl|g_j(\widetilde a_r,b_r)-g_j(a_r,b_r)\bigr|
 <\frac{\eta}{3}
 \qquad\text{for every }j.
\]
We henceforth write $a_r$ for these perturbed coordinates.
If the initial atomic measure is zero, replace it by $c\delta_x$,
where $x\in\Omega$ and $c>0$ is chosen so that
$c|g_j(x)|<\eta/3$ for every $j$.
Thus we may assume $N\geq1$, with pairwise distinct first
coordinates, at an additional error of less than $\eta/3$
for each test function.

For each $r$, the fiber
\[
 U_r:=\{z'\in B_R:(a_r,z')\in\Omega\}
\]
contains a ball $B(b_r,\rho_r)$ for some $\rho_r>0$.
For each test function $g_j$, the function $z'\mapsto g_j(a_r,z')$,
extended by zero outside $U_r$, is continuous with compact support in
$\C^m$.
Set $\nu_{r,\delta}=(dd^cH_{b_r,\delta})^m$.
Since $B(b_r,\rho_r)\subset U_r$,
Lemma~\ref{lem:logarithmic-potential} with mass $c_r$ and center $b_r$ gives
\[
 \int_{U_r}g_j(a_r,z')\,d\nu_{r,\delta}(z')
 \longrightarrow c_rg_j(a_r,b_r),\qquad
 \nu_{r,\delta}(\C^m\setminus U_r)
 \leq \frac{c_rm\delta^2}{\rho_r^2}.
\]
There are only finitely many slices and test functions, so we may
choose a common $\delta>0$ such that
\[
 \sum_{r=1}^N
 \left|
 \int_{U_r}g_j(a_r,z')\,d\nu_{r,\delta}(z')
 -c_rg_j(a_r,b_r)
 \right|<\frac{\eta}{3}
 \qquad\text{for every }j.
\]
For each $r$, choose the constant $C_r$ supplied by the lemma so that
$W_r:=H_{b_r,\delta}+C_r>0$ and $\log W_r$ is plurisubharmonic on a neighborhood of $\overline{B_R}$.
Adding $C_r$ does not change $(dd^cH_{b_r,\delta})^m$.
Hence the measure $\sigma$ defined by \eqref{eq:slice-measure}
satisfies
\[
 \left|\int_\Omega g_j\,d\sigma-\int_\Omega g_j\,d\mu\right|
 <\eta
 \qquad\text{for every }j,
\]
by combining the three preceding error estimates.
Since $\mu$, the finite family of test functions, and $\eta>0$
were arbitrary, these slice measures are vaguely dense in
$\cM_+(\Omega)$.

For $r=1,\ldots,N$, define the Lagrange polynomial
\[
 L_r(\zeta):=
 \prod_{\substack{1\leq s\leq N\\s\ne r}}
 \frac{\zeta-a_s}{a_r-a_s},
\]
with the empty product interpreted as $1$. Then
$L_r(a_s)=\delta_{rs}$ and $\sum_{r=1}^N L_r\equiv1$. Put
\[
 V_0(z_1,z'):=\sum_{r=1}^N|L_r(z_1)|^2W_r(z'),
\]
 then $V_0>0$ is smooth and psh on a neighborhood of $\overline\Omega$.  On
$D_s=\{z_1=a_s\}\cap\Omega$ one has
$V_0|_{D_s}=W_s$, and the components $D_s$ are disjoint.  This gives
\eqref{eq:programmed-slice}.  Finally,
\eqref{eq:normalizing-radius} implies $|F|\leq1$ on $\Omega$.
\end{proof}

\subsection{Control of the potential}
\label{sec:programming}

Having constructed the slice measure, we next control the size of
the perturbation in $L^p(\Omega)$. The following estimate bounds
the logarithmic potential uniformly in the slice locations.
For $N$ slices, it gives a bound of the form $C_{\Omega,p}N$,
so multiplying the potential by $\varepsilon$ makes its
$L^p$ norm small whenever $\varepsilon N$ is small.

\begin{lemma}[Uniform logarithmic $L^p$ estimate]
\label{lem:uniform-log-Lp}
Let $\Omega\subset\C^n$ ($n\geq2$) be a bounded domain, let
$1\leq p<\infty$, and let $R_1$ satisfy
\eqref{eq:normalizing-radius}.  There is $C_{\Omega,p}<\infty$ such that
\[
 \left\|\log\left|\frac{z_1-a}{R_1}\right|^2
 \right\|_{L^p(\Omega)}\leq C_{\Omega,p}
\]
for every $a$ in the $z_1$-projection of $\Omega$.
\end{lemma}

\begin{proof}
Write $\Delta_r=\{\zeta\in\C:|\zeta|<r\}$.  Choose $R_0>0$ and a
bounded ball $B\subset\C^{n-1}$ such that
$\Omega\subset\Delta_{R_0}\times B$.  If
$a$ belongs to the $z_1$-projection of $\Omega$, then $|a|<R_0$, and the
translation $w=z_1-a$ maps $\Delta_{R_0}$ into the fixed disc
$\Delta_{2R_0}$.  Fubini's theorem gives
\begin{align*}
 &\int_\Omega
 \left|\log\left|\frac{z_1-a}{R_1}\right|^2\right|^p
 d\lambda_{2n}\leq\lambda_{2n-2}(B)
 \int_{|w|<2R_0}
 \left|\log\left|\frac{w}{R_1}\right|^2\right|^p d\lambda_2(w)\\
 &\qquad=2\pi\lambda_{2n-2}(B)
 \int_0^{2R_0}|2\log(r/R_1)|^p r\,dr.
\end{align*}
The final integral is finite, and  is independent of $a$.
\end{proof}

\subsection{Truncations and control of lower-degree currents}

Prescribing the top-degree measure alone does not control the mixed
terms that arise on adding a psh function.  We therefore use two scales
to retain the slice measure while making every lower-degree current
small.  The resulting trace estimates will control those mixed terms.

\begin{lemma}[Limits of logarithmic truncations]\label{thm:awnw}
Let $\Omega\subset\C^n$ ($n\geq2$) be a bounded domain, and let
$D,F,V_0$ be the data in Proposition~\ref{prop:slices}.  Put $m=n-1$.
For $\varepsilon>0$, set
\[
 q_\varepsilon=\varepsilon\log|F|^2,\qquad
 v_\varepsilon=\varepsilon^{-1/m}V_0,\qquad
 h_{\varepsilon,A}=\max\{q_\varepsilon,v_\varepsilon-A\}.
\]
For fixed $\varepsilon$ and every $1\leq r\leq n$,
\begin{equation}\label{eq:all-degree-limit}
 (\ddc h_{\varepsilon,A})^r
 \longrightarrow
 \varepsilon^{(n-r)/(n-1)}[D]\wedge(\ddc V_0)^{r-1}
 \qquad(A\to+\infty)
\end{equation}
as currents.  If $A\geq\sup_\Omega v_\varepsilon$, then
$h_{\varepsilon,A}$ is continuous, bounded, and psh, with
\begin{equation}\label{eq:floor-choice}
 \varepsilon\log|F|^2\leq h_{\varepsilon,A}\leq0.
\end{equation}
\end{lemma}

\begin{proof}
Fix $\varepsilon>0$ and write
\[
 q=q_\varepsilon=\varepsilon\log|F|^2\leq 0,
 \qquad v=v_\varepsilon=\varepsilon^{-1/(n-1)}V_0,
 \qquad h_A=\max\{q,v-A\}\geq -A.
\]
If $A\geq\sup_\Omega v$, then $h_A$ is bounded and psh. Moreover, it is continuous on
$\Omega\setminus D$.
At each point of $D$, the logarithmic branch tends to $-\infty$,
whereas $v-A\geq -A$. Hence $h_A=v-A$ in a
neighborhood of that point. This proves continuity throughout
$\Omega$ and establishes \eqref{eq:floor-choice}.

We next verify the hypotheses of the truncation theorem \cite{AWNW}*{Theorem~1.5}.
The function $q$ is locally integrable and  pluriharmonic on
$\Omega\setminus D$ by definition of $F$.

For $1\leq r\leq n$, the non-pluripolar product is defined by
\[
 \langle\ddc q\rangle^r
 =\lim_{k\to\infty}
 \mathbf1_{\{q>-k\}}
 \bigl(\ddc\max\{q,-k\}\bigr)^r.
\]
The set
\[
 \{q>-k\}=\{|F|^2>e^{-k/\varepsilon}\}
\]
is open and disjoint from $D$. On this set,
$\max\{q,-k\}=q$ and $q$ is smooth and pluriharmonic.
Bedford--Taylor locality therefore gives
\[
 \mathbf1_{\{q>-k\}}
 \bigl(\ddc\max\{q,-k\}\bigr)^r=0.
\]
Consequently,
\[
 \langle\ddc q\rangle^r=0
 \qquad(1\leq r\leq n).
\]
Together with $q\in L^1_{\loc}(\Omega)$, this
verifies $q\in\mathcal G(\Omega)$ as defined in
\cite{AWNW}*{Definition~1.1}.

Using the notation of \cite{AWNW}*{Definition~1.2}, set
\[
 [\ddc q]^r
 :=\ddc\bigl(q\langle\ddc q\rangle^{r-1}\bigr),
 \qquad
 S_r(q):=[\ddc q]^r-\langle\ddc q\rangle^r,
 \qquad
 \langle\ddc q\rangle^0=1.
\]
Poincar\'e--Lelong, with the normalization
\eqref{eq:PL-normalization}, gives
\[
 [\ddc q]^1=S_1(q)=\varepsilon[D].
\]
For $r\geq2$, the factor
$\langle\ddc q\rangle^{r-1}$ is zero, so
\[
 [\ddc q]^r=S_r(q)=0.
\]

Since $q\in\mathcal G(\Omega)$ and $v$ is smooth and psh,
\cite{AWNW}*{Theorem~1.5} applies directly to
$h_A=\max\{q,v-A\}$. For each $1\leq r\leq n$, it yields
\begin{equation}\label{eq:awnw}
 (\ddc h_A)^r
 \longrightarrow
 [\ddc q]^r
 +\sum_{j=1}^{r-1}
 S_j(q)\wedge(\ddc v)^{r-j}
 \qquad(A\to+\infty)
\end{equation}
as currents. For $r=1$, the sum is empty and the limit is
$\varepsilon[D]$. For $r\geq2$, only the term $j=1$ remains,
and hence
\begin{align*}
 (\ddc h_A)^r
 &\longrightarrow
 \varepsilon[D]\wedge(\ddc v)^{r-1}=\varepsilon^{\,1-(r-1)/(n-1)}
   [D]\wedge(\ddc V_0)^{r-1}=\varepsilon^{(n-r)/(n-1)}
   [D]\wedge(\ddc V_0)^{r-1}.
\end{align*}
This also agrees with the case $r=1$, proving
\eqref{eq:all-degree-limit}.
\end{proof}

Put $\beta=\ddc|z|^2$.  The trace measure of a positive $(r,r)$-current
$T$ will mean $T\wedge\beta^{n-r}$.

\begin{lemma}[A perturbation with small lower-degree currents]
\label{lem:small-perturbation}
Let $\Omega\subset\C^n$ ($n\geq2$) be a bounded domain.  Fix a slice
measure $\sigma$ from Proposition~\ref{prop:slices}, a compact set
$L\subset\subset\Omega$, test functions
$g_1,\ldots,g_s\in C_c^\infty(\Omega)$, and numbers
$\eta,\gamma,t>0$.  There is a continuous bounded psh function
$h\leq0$ such that
\begin{align}
 \|h\|_{L^p(\Omega)}&\leq C_{\Omega,p}t
       &&(1\leq p<\infty),\label{eq:perturbation-Lp}\\
 \left|\int g_j\,d(\MA(h)-\sigma)\right|&<\eta
       &&(1\leq j\leq s),\label{eq:perturbation-top}\\
 \int_L(\ddc h)^r\wedge\beta^{n-r}&<\gamma
       &&(1\leq r<n).\label{eq:perturbation-lower}
\end{align}
The same $h$ gives the estimate for every finite $p$.
\end{lemma}

\begin{proof}
Fix the data $D,F,V_0$ and their number of slices $N$.  Choose
$\chi\in C_c^\infty(\Omega)$, $\chi\geq0$, equal to one near $L$.
For $1\leq r<n$, the number
\[
 M_r=\int_\Omega\chi[D]\wedge(\ddc V_0)^{r-1}\wedge\beta^{n-r}
\]
is finite.  Choose $\varepsilon>0$ so that
\begin{equation}\label{eq:epsilon-choice}
 \varepsilon N\leq t,\qquad
 \varepsilon^{(n-r)/(n-1)}M_r<\gamma/2
 \quad(1\leq r<n).
\end{equation}
At this fixed $\varepsilon$, \eqref{eq:all-degree-limit} in top degree
has limit $\sigma$ by Proposition~\ref{prop:slices}.  Lemma~\ref{thm:awnw} allows us to choose
a single $A\geq\sup_\Omega v_\varepsilon$ such that
\[
 \left|
 \int_\Omega g_j\,(\ddc h_{\varepsilon,A})^n
 -\int_\Omega g_j\,d\sigma
 \right|<\eta
 \qquad\text{for every }j,
\]
and, simultaneously, for every $1\leq r<n$,
\[
 \left|
 \int_\Omega\chi\,(\ddc h_{\varepsilon,A})^r\wedge\beta^{n-r}
 -\varepsilon^{(n-r)/(n-1)}
  \int_\Omega\chi\,[D]\wedge(\ddc V_0)^{r-1}
       \wedge\beta^{n-r}
 \right|<\frac{\gamma}{2}.
\]
The second integral, including its $\varepsilon$ factor, is
less than $\gamma/2$ by our choice of $\varepsilon$.
Since $\chi\geq0$ and $\chi=1$ near $L$, positivity gives
\[
 \int_L(\ddc h_{\varepsilon,A})^r\wedge\beta^{n-r}
 \leq
 \int_\Omega\chi\,(\ddc h_{\varepsilon,A})^r\wedge\beta^{n-r}
 <\gamma.
\]
Thus $h=h_{\varepsilon,A}$ satisfies
\eqref{eq:perturbation-top}--\eqref{eq:perturbation-lower}.
Finally, \eqref{eq:floor-choice}, Minkowski's inequality, and
Lemma~\ref{lem:uniform-log-Lp} give, for every finite $p$,
\begin{equation}\label{eq:global-log-control}
 \|h\|_{L^p(\Omega)}
 \leq\varepsilon\sum_{r=1}^N
 \left\|\log\left|\frac{z_1-a_r}{R_1}\right|^2\right\|_{L^p(\Omega)}
 \leq C_{\Omega,p}\varepsilon N\leq C_{\Omega,p}t.
\end{equation}
The slice data are fixed first, then $\varepsilon$, and finally $A$.
\end{proof}

\section{Addition of measures}\label{sec:addition-construction}

The lower-degree trace estimates control the mixed terms produced by
adding the perturbation to a bounded psh function.  In the estimate below,
the constant depends on the bound of the given function and is
independent of the size of the perturbation.

\begin{lemma}[Mixed Chern--Levine--Nirenberg estimate]\label{lem:CLN-mixed}
Let $\Omega\subset\C^n$ ($n\geq2$) be a domain, let
$K\subset L^\circ$ with $K,L\subset\subset\Omega$, and let $1\leq r<n$.
If $b,h$ are locally bounded psh functions, then
\begin{equation}\label{eq:CLN-mixed}
 \int_K(\ddc b)^{n-r}\wedge(\ddc h)^r
 \leq C_{K,L,n}\|b\|_{L^\infty(L)}^{n-r}
       \int_L(\ddc h)^r\wedge\beta^{n-r},
\end{equation}
 where $C_{K,L,n}>0$ depends only on $K$, $L$, and $n$.
\end{lemma}

\begin{proof}
This is the Chern--Levine--Nirenberg estimate for the positive closed
current $T=(\ddc h)^r$; see
\cite{DemaillyCADG}*{Chapter~III, (3.3)}.
\end{proof}

\begin{proposition}[Addition for finitely many test functions]\label{prop:programming}
Let $\Omega\subset\C^n$ ($n\geq2$) be a bounded domain, let
$b\in\PSH(\Omega)\cap L^\infty(\Omega)$, and let $\sigma$ be a slice
measure from Proposition~\ref{prop:slices}.  Given
$g_1,\ldots,g_s\in C_c^\infty(\Omega)$ and $\eta,t>0$, there is a
continuous bounded psh $h\leq0$ such that $u=b+h$ satisfies
\begin{align}
 u&\leq b,\qquad
 \|u-b\|_{L^p(\Omega)}\leq C_{\Omega,p}t
       \quad(1\leq p<\infty),\label{eq:program-Lp}\\
 \left|\int g_j\,d\bigl(\MA(u)-\MA(b)-\sigma\bigr)\right|
 &<\eta\quad(1\leq j\leq s).\label{eq:program-measure}
\end{align}
If $b$ is continuous, so is $u$.
\end{proposition}

\begin{proof}
Choose compact sets $K,L\Subset\Omega$ such that
\[
 \bigcup_j\operatorname{supp}g_j\subset K\subset L^\circ.
\]
Set
\[
 B=\|b\|_{L^\infty(L)},
 \qquad G=\max_j\|g_j\|_\infty,
\]
and choose $\gamma>0$ so that
\[
 G C_{K,L,n}\gamma
 \sum_{r=1}^{n-1}\binom nr B^{n-r}<\frac{\eta}{2}.
\]
Apply Lemma~\ref{lem:small-perturbation} with the compact set $L$,
lower-degree tolerance $\gamma$, top-degree tolerance $\eta/2$,
and the prescribed $t$. It gives a continuous bounded psh
function $h\leq0$ such that
\[
 \int_L(\ddc h)^r\wedge\beta^{n-r}<\gamma
 \qquad(1\leq r<n),\qquad
 \left|\int_\Omega g_j\,\MA(h)-\int_\Omega g_j\,d\sigma\right|
 <\frac{\eta}{2}
 \qquad\text{for every }j,
\]
together with the stated $L^p$ bounds.

Set $u=b+h$. Since both $b$ and $h$ are bounded and psh,
Bedford--Taylor multilinearity gives
\begin{equation}\label{eq:BT-addition-expansion}
 \MA(u)=\MA(b)+\MA(h)
 +\sum_{r=1}^{n-1}\binom nr
       (\ddc b)^{n-r}\wedge(\ddc h)^r.
\end{equation}
Each mixed measure is nonnegative. Since
$\operatorname{supp}g_j\subset K$, the estimate
\eqref{eq:CLN-mixed} yields
\begin{align*}
 &\left|
 \sum_{r=1}^{n-1}\binom nr
 \int_\Omega g_j\,
       (\ddc b)^{n-r}\wedge(\ddc h)^r
 \right|\leq
 G\sum_{r=1}^{n-1}\binom nr
 \int_K(\ddc b)^{n-r}\wedge(\ddc h)^r\\
 &\qquad\leq
 G C_{K,L,n}
 \sum_{r=1}^{n-1}\binom nr B^{n-r}
 \int_L(\ddc h)^r\wedge\beta^{n-r}\leq
 G C_{K,L,n}\gamma
 \sum_{r=1}^{n-1}\binom nr B^{n-r}
 <\frac{\eta}{2}.
\end{align*}
Combining this with the top-degree estimate and
\eqref{eq:BT-addition-expansion}, we obtain
\[
 \left|
 \int_\Omega g_j\,\MA(u)
 -\int_\Omega g_j\,\MA(b)
 -\int_\Omega g_j\,d\sigma
 \right|<\eta
 \qquad\text{for every }j.
\]

Finally, $u\leq b$ because $h\leq0$, and
\[
 \|u-b\|_{L^p(\Omega)}
 =\|h\|_{L^p(\Omega)}
 \leq C_{\Omega,p}t
 \qquad(1\leq p<\infty).
\]
If $b$ is continuous, then $u$ is continuous as well.
\end{proof}

\subsection{Addition for sequences of psh functions}
\label{sec:notation}

To impose these estimates along one sequence, we fix a countable family
of test functions that determines the vague topology.  The associated
metric also compares two moving measure sequences without requiring
common local mass bounds.

\begin{lemma}[Metrization of the vague topology]
\label{lem:vague-metrization}
On every domain $\Omega\subset\C^n$, there are real-valued functions
$(f_j)_{j\geq1}\subset C_c^\infty(\Omega)$ and compact sets $K_j$ such that
\begin{equation}\label{eq:exhaustion}
 K_1\subset\subset K_2^\circ\subset\subset\cdots,\qquad
 \bigcup_{j\geq1}K_j=\Omega,\qquad
 \bigcup_{\ell\leq j}\supp f_\ell\subset K_j^\circ.
\end{equation}
The family $(f_j)$ determines vague convergence on $\cM_+(\Omega)$, and
\begin{equation}\label{eq:vague-metric}
 d_{\mathrm v}(\mu,\nu)
 :=\sum_{j=1}^\infty2^{-j}
 \frac{\left|\int_\Omega f_j\,d(\mu-\nu)\right|}
 {1+\left|\int_\Omega f_j\,d(\mu-\nu)\right|}
\end{equation}
is a metric inducing the vague topology.  Moreover, this topology is separable.
\end{lemma}

\begin{proof}
Choose relatively compact open sets $U_m$ exhausting $\Omega$, with
$\overline U_m\subset U_{m+1}$, and put $L_m=\overline U_{m+1}$.
Let $\mathcal D_m\subset C_c^\infty(U_{m+1};\mathbb R)$ be a
countable uniformly dense subset of $C_0(U_{m+1};\mathbb R)$,
the space of continuous functions vanishing at the boundary and
at infinity. Such a family exists by separability and smooth
uniform approximation. Extend its members by zero to $\Omega$.
Choose $\chi_m\in C_c^\infty(\Omega)$, $\chi_m\geq0$, equal to
one near $L_m$. Enumerate all the $\mathcal D_m$ and $\chi_m$ as
$(f_j)$, and enlarge an arbitrary compact exhaustion to obtain
$(K_j)$ satisfying \eqref{eq:exhaustion}.

Convergence of the $\chi_m$ coordinates of positive measures
$\mu_k$ to those of $\mu$ gives the local mass bound
\[
 \sup_k\mu_k(L_m)\leq\sup_k\int\chi_m\,d\mu_k<\infty.
\]
For real $g\in C_c(\Omega)$, choose $m$ with $\supp g\subset U_m$
and approximate $g$ uniformly by $f\in\mathcal D_m$. Both functions
are supported in $L_m$, and
\[
 \left|\int g\,d(\mu_k-\mu)\right|
 \leq\left|\int f\,d(\mu_k-\mu)\right|
 +\|g-f\|_\infty\bigl(\mu_k(L_m)+\mu(L_m)\bigr).
\]
Thus coordinate convergence implies vague convergence; complex
functions are treated by their real and imaginary parts. The same
estimate, with the cutoff coordinate restricted to a bounded
neighborhood of its value at $\mu$, shows that each $C_c$ integral
is continuous in the coordinate topology. Hence that topology is
exactly the vague topology.

The function $t\mapsto t/(1+t)$ is increasing and subadditive on
$[0,\infty)$, so \eqref{eq:vague-metric} satisfies the triangle
inequality; the determining property separates measures. Its finite
partial sums and uniformly bounded tails show that it induces the
coordinate topology. Finally, finite atomic measures supported on
a fixed countable dense subset of $\Omega$, with nonnegative rational
weights, form a countable vague dense set: on the compact supports
of finitely many test functions, approximate a Radon measure by
finitely many atoms, then approximate their locations and weights.
\end{proof}

Fix one family, exhaustion, and metric supplied by
Lemma~\ref{lem:vague-metrization} for each domain under consideration.

For two sequences $(\alpha_k),(\beta_k)\subset\cM_+(\Omega)$,
$d_{\mathrm v}(\alpha_k,\beta_k)\to0$ is equivalent to
\[
 \int f_j\,d(\alpha_k-\beta_k)\longrightarrow0
 \quad\text{for every fixed }j.
\]
It follows directly from compatibility of the metric that the two positive
sequences have the same vague subsequential limits.  Also,
\begin{equation}\label{eq:translation-invariance}
 d_{\mathrm v}(\alpha+\gamma,\beta+\gamma)
 =d_{\mathrm v}(\alpha,\beta)
\end{equation}
for positive Radon measures $\alpha,\beta,\gamma$.

For these two sequences, convergence of their metric distance to zero
need not imply convergence of their signed differences against every
continuous compactly supported test function.
Example~\ref{ex:metric-dependence} records the distinction, and
Proposition~\ref{prop:mass-control} gives the stronger conclusion under
local mass bounds.

\begin{theorem}[Approximation of a sum of measures]\label{thm:tracking}
Let $\Omega\subset\C^n$ ($n\geq2$) be a bounded domain, and fix
$(f_j)$, $(K_j)$, and $d_{\mathrm v}$ as above.  Let
\[
 b_k\in\PSH(\Omega)\cap L^\infty(\Omega),\qquad
 \mu_k\in\cM_+(\Omega)
\]
be arbitrary sequences.  There are bounded psh functions $u_k$ such that
\begin{align}
 u_k&\leq b_k\quad\text{on }\Omega,\label{eq:background-order}\\
 \|u_k-b_k\|_{L^p(\Omega)}&\longrightarrow0
 \quad\text{for every }1\leq p<\infty,\label{eq:background-Lp}\\
 d_{\mathrm v}\bigl(\MA(u_k),\MA(b_k)+\mu_k\bigr)
 &\longrightarrow0.\label{eq:background-measure}
\end{align}
The same sequence works for every finite $p$.  If every $b_k$ is continuous,
the $u_k$ can be chosen smooth.
\end{theorem}

\begin{proof}
For each $k$, Proposition~\ref{prop:slices} gives a slice measure
$\sigma_k$ satisfying
\begin{equation}\label{eq:slice-approximation}
 \left|\int_\Omega f_j\,d(\sigma_k-\mu_k)\right|<1/k
 \qquad(1\leq j\leq k).
\end{equation}
Only finitely many compactly supported test functions are involved,
so no finite total mass assumption is needed. Apply
Proposition~\ref{prop:programming} to $b_k$, $\sigma_k$, these test
functions, and $\eta=t=1/k$. It gives a continuous bounded psh
function $h_k\leq0$ such that $w_k=b_k+h_k$ satisfies
\[
 w_k\leq b_k,\qquad
 \|w_k-b_k\|_{L^p(\Omega)}\leq C_{\Omega,p}/k
 \quad(1\leq p<\infty),
\]
and, after combining the addition error with
\eqref{eq:slice-approximation},
\begin{equation}\label{eq:finite-coordinate-control}
 \left|\int_\Omega f_j\,d\bigl(\MA(w_k)-\MA(b_k)-\mu_k\bigr)\right|
 <2/k\qquad(1\leq j\leq k).
\end{equation}
Each $w_k$ is globally bounded. The parameters may depend on
$\|b_k\|_\infty$; no common bound is used.
If smoothness is not required, take $u_k=w_k$.

If every $b_k$ is continuous, put
$C_\Omega=1+\sup_\Omega|z|^2$. Richberg approximation
\cite{DemaillyCADG}*{Chapter~I, Theorem~5.21} gives smooth strictly
psh functions $v_{k,\delta}$ with
\[
 w_k+\delta|z|^2\leq v_{k,\delta}
 \leq w_k+\delta(|z|^2+1).
\]
For fixed $k$, these functions converge uniformly to $w_k$ as
$\delta\downarrow0$. Bedford--Taylor continuity therefore allows
us to choose $0<\delta_k\leq1/k$ such that
\[
 \left|\int_\Omega f_j\,d\bigl(\MA(v_{k,\delta_k})-\MA(w_k)\bigr)\right|
 <1/k\qquad(1\leq j\leq k).
\]
Set $u_k=v_{k,\delta_k}-\delta_kC_\Omega$. The constant shift
preserves its Monge--Amp\`ere measure, and
\[
 w_k-\delta_kC_\Omega\leq u_k\leq w_k\leq b_k.
\]
Thus $u_k$ is globally bounded and, for each $1\leq p<\infty$,
\[
 \|u_k-b_k\|_{L^p(\Omega)}
 \leq C_{\Omega,p}/k+
 \delta_kC_\Omega\lambda_{2n}(\Omega)^{1/p}\longrightarrow0.
\]
The choices are independent of $p$, so one sequence gives all these
limits. In either case the first $k$ measure coordinates have error
at most $3/k$, and hence
\[
 d_{\mathrm v}\bigl(\MA(u_k),\MA(b_k)+\mu_k\bigr)
 \leq\frac3k\sum_{j=1}^k2^{-j}+\sum_{j>k}2^{-j}
 \leq\frac3k+2^{-k}\longrightarrow0.
\]
\end{proof}

\begin{remark}\label{rem:tracking-meaning}
In particular, if $\MA(b_k)+\mu_k$ converges vaguely to a measure $\nu$,
then $\MA(u_k)\to\nu$ vaguely.  For arbitrary $\mu_k$, the two sequences
in \eqref{eq:background-measure} have the same full vague cluster set.
\end{remark}

\begin{proof}[Proof of Theorem~\ref{thm:fixed-addition}]
Apply Theorem~\ref{thm:tracking} with $b_k=\varphi$ and $\mu_k=\mu$.
The measure $\MA(\varphi)+\mu$ is independent of $k$, so metric convergence gives
$\MA(u_k)\to\MA(\varphi)+\mu$ vaguely.  The order, all finite $L^p$
limits, and the smooth case also follow from that theorem.
\end{proof}

\begin{corollary}[Smooth approximation of zero]
\label{thm:zero-smooth}
Let $\Omega\subset\C^n$ ($n\geq2$) be a bounded domain.  For every sequence
$(\mu_k)\subset\cM_+(\Omega)$ there are smooth bounded psh functions
$u_k\leq0$ such that
\[
 d_{\mathrm v}(\MA(u_k),\mu_k)\to0,\qquad
 \|u_k\|_{L^p(\Omega)}\to0\quad(1\leq p<\infty).
\]
\end{corollary}

\begin{proof}
Take $b_k=0$ in Theorem~\ref{thm:tracking}.
\end{proof}
\section{Reachable limits and full cluster sets}\label{sec:tracking-results}

To pass from individual limits to full cluster sets, we prove that the
reachable set is closed and use diagonal selection to realize its
nonempty closed subsets.  A sequence with a growing point mass realizes
the empty cluster set.

For a sequence in a metric space $X$, write $\Clust_X$ for its set of
subsequential limits.  We continue to use $\Clust_{\mathrm v}$ for the
space $\cM_+(\Omega)$.

\begin{lemma}[Sequential cluster sets]\label{lem:cluster-topology}
Let $X$ be a separable metric space.  Every nonempty closed subset
$F\subset X$ is the cluster set of a sequence in $F$.
\end{lemma}

\begin{proof}
Choose a countable dense sequence $x_1,x_2,\ldots$ in $F$, allowing
repetitions, and concatenate the finite lists
$(x_1),(x_1,x_2),(x_1,x_2,x_3),\ldots$.
Every point of $F$ is approached along a subsequence, since each
$x_j$ occurs arbitrarily late, and no point outside $F$ is a
subsequential limit because all terms lie in the closed set $F$.
\end{proof}

\begin{proof}[Proof of Theorem~\ref{thm:structure}]
First let $\nu\in\mathscr R_b(\Omega,\varphi)$, and choose bounded
psh functions $b_k\to\varphi$ in $L^1_{\loc}(\Omega)$ with
$\MA(b_k)\to\nu$ vaguely. For $\mu\in\cM_+(\Omega)$,
Theorem~\ref{thm:tracking} gives bounded psh functions $u_k$ with
\[
 \|u_k-b_k\|_{L^1(\Omega)}\to0,\qquad
 d_{\mathrm v}\bigl(\MA(u_k),\MA(b_k)+\mu\bigr)\to0.
\]
Thus $u_k\to\varphi$ locally in $L^1$ and
$\MA(u_k)\to\nu+\mu$ vaguely, proving addition in the bounded
class. If $\nu\in\mathscr R_{\mathrm{sm}}(\Omega,\varphi)$,
choose $b_k$ smooth and use the smooth part of the same theorem.
This proves addition in the smooth class as well.

Fix $a\in\{b,\mathrm{sm}\}$. Since $0\in\cM_+(\Omega)$, addition
gives $\mathscr R_a+\cM_+=\mathscr R_a$.
If $0\in\mathscr R_a$, every $\mu\in\cM_+$ belongs to
$\mathscr R_a$ because $\mu=0+\mu$; conversely,
$\mathscr R_a=\cM_+$ includes zero.

We use one diagonal selection for both closure and realization.
For any sequence $\nu_k\in\mathscr R_a(\Omega,\varphi)$, select
from a sequence realizing $\nu_k$ one function $u_k$ in the same
class such that
\[
 \|u_k-\varphi\|_{L^1(K_k)}<1/k,\qquad
 d_{\mathrm v}(\MA(u_k),\nu_k)<1/k.
\]
Then $u_k\to\varphi$ in $L^1_{\loc}(\Omega)$, and the two measure
sequences have the same limits along corresponding subsequences.
In particular, if $\nu_k\to\nu$ vaguely, this selection realizes
$\nu$, proving that $\mathscr R_a$ is closed in the metrizable
vague topology.

Every subsequential measure limit of an approximating sequence
belongs to $\mathscr R_a$, because the corresponding potentials
still converge to $\varphi$ in the same class. Its full cluster
set is closed: in a metric space it is the intersection of the
closures of all tails of the sequence. Conversely, for a nonempty
vaguely closed set $F\subset\mathscr R_a$,
Lemma~\ref{lem:cluster-topology} supplies a sequence $\nu_k\in F$
with full cluster set $F$. The preceding diagonal selection
realizes exactly the same full cluster set.

For the empty set, let $b_k\to\varphi$ in $L^1_{\loc}(\Omega)$
be any bounded psh approximating sequence in the class under
consideration. Fix $x_0\in\Omega$ and put
$\beta_k=\MA(b_k)+k\delta_{x_0}$. Choose $g\in C_c(\Omega)$,
$g\geq0$, with $g(x_0)=1$. Positivity gives
\[
 \int_\Omega g\,d\beta_k\geq k,
\]
so $(\beta_k)$ has no vaguely convergent subsequence. Applying
Theorem~\ref{thm:tracking} with $\mu_k=k\delta_{x_0}$ produces
$u_k\to\varphi$ locally in $L^1$, with
$d_{\mathrm v}(\MA(u_k),\beta_k)\to0$. The measure sequences
therefore have the same empty cluster set.
In the bounded class take $b_k=\varphi$; this also gives convergence
to $\varphi$ in every finite global $L^p$ norm. In the smooth class,
use any smooth bounded approximating sequence, whose existence
follows in particular from $\mathscr R_{\mathrm{sm}}\ne\varnothing$.
No convergence of $\MA(b_k)$ is needed in this construction.
\end{proof}

\begin{remark}[The empty set in the smooth class]\label{rem:smooth-empty}
A smooth bounded psh approximating sequence with empty measure
cluster set exists if and only if there is any smooth bounded psh
sequence $b_k\to\varphi$ in $L^1_{\loc}(\Omega)$. Necessity is
immediate. For sufficiency, the last construction in the proof
of Theorem~\ref{thm:structure} applies to this sequence without
assuming convergence of $\MA(b_k)$.
Thus $\mathscr R_{\mathrm{sm}}(\Omega,\varphi)=\varnothing$ excludes
convergent measure subsequences, but does not by itself exclude
smooth bounded psh approximating sequences with empty measure
cluster set.
\end{remark}

The geometric part of the proof must establish
$0\in\mathscr R_b(\Omega,\varphi)$, that is, the existence of
\begin{equation}\label{eq:zero-approximation}
 \begin{gathered}
 b_j\in\PSH(\Omega)\cap L^\infty(\Omega),\qquad
 b_j\to\varphi\quad\text{in }L^1_{\loc},\\
 \MA(b_j)\to0\quad\text{vaguely}.
 \end{gathered}
\end{equation}
Continuous approximants suffice for the smooth classification.  The
following lemma and corollary record this fact and the corresponding
construction for moving measures.

\begin{lemma}[Smoothing continuous approximants]\label{lem:continuous-approximation}
Let $\Omega\subset\C^n$ ($n\geq2$) be a bounded domain, and let bounded
continuous psh functions $b_k$ satisfy $b_k\to\varphi$ in $L^1_{\loc}$
and $\MA(b_k)\to\nu$ vaguely.  Then $\nu\in\mathscr R_{\mathrm{sm}}(\Omega,\varphi)$.
In particular, the existence of continuous $b_j$ in
\eqref{eq:zero-approximation} is equivalent to
$0\in\mathscr R_{\mathrm{sm}}(\Omega,\varphi)$.
\end{lemma}

\begin{proof}
Apply the smooth part of Theorem~\ref{thm:tracking} to $b_k$
with $\mu_k=0$. Since each $b_k$ is continuous, it gives
smooth bounded psh functions $u_k$ satisfying
\[
 \|u_k-b_k\|_{L^1(\Omega)}\longrightarrow0,
 \qquad
 d_{\mathrm v}\bigl(\MA(u_k),\MA(b_k)\bigr)\longrightarrow0.
\]
Thus $u_k\to\varphi$ in $L^1_{\loc}(\Omega)$.
Moreover, $\MA(b_k)\to\nu$ vaguely by assumption, so
\[
 d_{\mathrm v}(\MA(u_k),\nu)
 \leq
 d_{\mathrm v}\bigl(\MA(u_k),\MA(b_k)\bigr)
 +d_{\mathrm v}(\MA(b_k),\nu)
 \longrightarrow0,
\]
which implies
$\MA(u_k)\to\nu$ vaguely. Therefore $(u_k)_{k\geq 1}$ realizes $\nu$
in the smooth class, proving
$\nu\in\mathscr R_{\mathrm{sm}}(\Omega,\varphi)$.

Taking $\nu=0$ shows that continuous bounded psh approximants
with measure limit zero imply
$0\in\mathscr R_{\mathrm{sm}}(\Omega,\varphi)$.
Conversely, if $0\in\mathscr R_{\mathrm{sm}}(\Omega,\varphi)$,
its realizing sequence consists of smooth, hence continuous,
bounded psh functions and satisfies
\eqref{eq:zero-approximation}.
\end{proof}

\begin{corollary}[Approximation of prescribed measures in the fixed metric]
\label{thm:moving-zero}
Let $\Omega\subset\C^n$ ($n\geq2$) be a bounded domain, and let
$\varphi\in\PSH(\Omega)\cap L^\infty(\Omega)$.
If $0\in\mathscr R_b(\Omega,\varphi)$, then for every sequence
$(\mu_k)_{k\geq 1}\subset\cM_+(\Omega)$ there are bounded psh functions
$u_k\to\varphi$ in $L^1_{\loc}(\Omega)$ such that
\[
 d_{\mathrm v}(\MA(u_k),\mu_k)\longrightarrow0.
\]
If $0\in\mathscr R_{\mathrm{sm}}(\Omega,\varphi)$, the functions can be
chosen smooth and bounded.
\end{corollary}

\begin{proof}
Choose $b_k$ satisfying \eqref{eq:zero-approximation}.  Theorem~\ref{thm:tracking} applied to
$b_k$ and $\mu_k$ gives
\begin{align*}
 d_{\mathrm v}(\MA(u_k),\mu_k)
 &\leq d_{\mathrm v}\bigl(\MA(u_k),\MA(b_k)+\mu_k\bigr)+d_{\mathrm v}\bigl(\MA(b_k)+\mu_k,\mu_k\bigr)\longrightarrow0.
\end{align*}
If $0\in\mathscr R_{\mathrm{sm}}(\Omega,\varphi)$, choose the $b_k$ smooth
and use the smooth part of Theorem~\ref{thm:tracking}.
\end{proof}

\section{Bounded approximation and the main classification}
\label{sec:geometry}

By Theorem~\ref{thm:structure}, the conclusions of
Theorem~\ref{thm:classification} for bounded approximants follow
once \eqref{eq:zero-approximation} is established.  For a bounded maximal $\varphi$,
the constant sequence suffices.  For a general bounded psh function,
Lelong's maximal approximants have zero Monge--Amp\`ere measure, but
need not be globally bounded.  The geometric conditions below allow us
to impose this bound while preserving the approximation on compact sets.

Lelong's approximation theorem states that, on a pseudoconvex domain in
$\C^n$ ($n\geq2$), continuous maximal psh functions are locally
$L^1$-dense among locally bounded psh functions; see \cite{Lelong} and the
precise restatement \cite{BacklundPersson}*{Theorem~2.1, p.~263}.
The approximants are defined on all of $\Omega$ and convergence is in
the local $L^1$ topology; the theorem supplies no global upper or lower bound.

If the continuous maximal approximants $\psi_j$ supplied on $\Omega$
are bounded above, global boundedness follows by lower truncation:
choose $N_j$ so that $\psi_j>-N_j$ near $K_j$ and put
$b_j=\max\{\psi_j,-N_j\}$.  Then $b_j$ is bounded and continuous,
agrees with $\psi_j$ near $K_j$, and has the same local approximation
error and vanishing Monge--Amp\`ere measure there.  An upper bound,
however, does not follow from local convergence.  On $\D^n$ ($n\geq2$),
where $\D=\{\zeta\in\C:|\zeta|<1\}$, the functions
\[
 \psi_j(z)=\max\left\{j^{-1}\log\frac1{|1-z_1|},-1\right\}
\]
are continuous, maximal, and locally uniformly convergent to zero,
but unbounded above.  Maximality follows because they depend on only
one variable.  The $\mathcal B_c$-convexity condition defined below
supplies the global boundedness that Lelong's theorem alone does not provide.

\subsection{Hulls and \texorpdfstring{$\mathcal B_c$}{Bc}-convexity}

Let $\Omega\subset\C^n$ be a bounded domain. Define
\[
 \mathcal B_c(\Omega)
 :=\PSH(\Omega)\cap C(\Omega)\cap L^\infty(\Omega),
\]
the class of continuous psh functions that are globally bounded on
$\Omega$. For a nonempty compact set $K\subset\subset\Omega$, its
\emph{$\mathcal B_c$-hull} is
\[
 \widehat K_{\mathcal B_c}
 :=\{z\in\Omega:b(z)\leq\sup_K b
       \text{ for every }b\in\mathcal B_c(\Omega)\}.
\]
Thus a point $z\in\Omega$ lies outside this hull precisely when some
$b\in\mathcal B_c(\Omega)$ satisfies $b(z)>\sup_K b$.
The hull contains $K$ and is closed relative to $\Omega$, since it is
an intersection of relatively closed sublevel sets of continuous functions.

We call $\Omega$ \emph{$\mathcal B_c$-convex} if
\begin{equation}\label{eq:compact-hulls}
 \widehat K_{\mathcal B_c}\subset\subset\Omega
 \qquad\text{for every nonempty compact }K\subset\subset\Omega.
\end{equation}
Since each hull is relatively closed, this means that it is compact
in $\Omega$. This is precisely the compact-hull condition referred to
in the Introduction.

The separating functions in the hull definition have two uses: they
produce a continuous psh exhaustion and allow a maximal approximant
to be made globally bounded without changing it near a prescribed
compact set.

\begin{lemma}[Separation and pseudoconvexity]\label{lem:Bc-pseudoconvex}
Let $\Omega\subset\C^n$ be a bounded $\mathcal B_c$-convex domain.
For every nonempty compact $K\subset\subset\Omega$ and every open set
$U\subset\subset\Omega$ containing $\widehat K_{\mathcal B_c}$, there
is $b\in\mathcal B_c(\Omega)$ such that
\begin{equation}\label{eq:Bc-separation}
 b\leq0\quad\text{on }K,\qquad
 b>1\quad\text{on }\partial U.
\end{equation}
Moreover, $\Omega$ admits a continuous psh exhaustion and is therefore
pseudoconvex.
\end{lemma}

\begin{proof}
For each $\xi\in\partial U$, the hull definition gives
$f_\xi\in\mathcal B_c(\Omega)$ with
$a_\xi:=f_\xi(\xi)-\sup_K f_\xi>0$. The normalized function
\[
 \widetilde f_\xi=\frac{2}{a_\xi}(f_\xi-\sup_K f_\xi)
\]
is nonpositive on $K$ and equals $2$ at $\xi$. By continuity it
exceeds $1$ on a neighborhood of $\xi$. Finitely many of these
neighborhoods cover $\partial U$. The maximum of the corresponding
functions belongs to $\mathcal B_c(\Omega)$ and satisfies
\eqref{eq:Bc-separation}.

Choose compact sets $K_j\subset K_{j+1}^\circ$ exhausting $\Omega$.
For each $j$, choose $U_j\subset\subset\Omega$ containing
$\widehat{K_j}_{\mathcal B_c}$ and a function $b_j$ satisfying
\eqref{eq:Bc-separation} for $K_j,U_j$. Define
\[
 \psi_j(z)=
 \begin{cases}
  \max\{b_j(z),0\},&z\in U_j,\\
  \max\{b_j(z),1\},&z\in\Omega\setminus U_j.
 \end{cases}
\]
Since $b_j>1$ near $\partial U_j$, the two expressions agree there.
Thus $\psi_j$ is continuous and psh, with
\[
 \psi_j\geq0,\qquad \psi_j|_{K_j}=0,\qquad
 \psi_j\geq1\quad\text{on }\Omega\setminus U_j.
\]
The sum $\tau=\sum_{j=1}^\infty\psi_j$ is locally finite. Indeed,
if $z\in K_J^\circ$, all terms with $j\geq J$ vanish throughout
$K_J^\circ$. Hence $\tau$ is a finite-valued continuous psh function.
For every positive integer $N$,
\[
 \{\tau<N\}\subset\bigcup_{j=1}^N U_j\subset\subset\Omega,
\]
because outside this union the first $N$ terms are each at least $1$.
These sublevel sets exhaust $\Omega$, so $\tau$ is an exhaustion.
The psh exhaustion criterion
\cite{DemaillyCADG}*{Chapter~I, Theorem~7.2(c)} implies that
$\Omega$ is pseudoconvex.
\end{proof}

The converse fails even for bounded domains, as the following example
shows.

\begin{example}[A pseudoconvex domain that is not $\mathcal B_c$-convex]
\label{ex:pseudoconvex-not-Bc}
Let $\D=\{w\in\C:|w|<1\}$ and $\D^*=\D\setminus\{0\}$. The domain
\[
 \Omega=\D^{n-1}\times(\D\setminus\{0\})\qquad(n\geq2)
\]
is pseudoconvex, being a product of planar domains.

Fix $0<r<1$ and put $K=\{0\}^{n-1}\times\{|w|=r\}$.
For every $b\in\mathcal B_c(\Omega)$, the bounded subharmonic
function $w\mapsto b(0,w)$ extends subharmonically across zero.
The maximum principle on $\{|w|<r\}$ gives
\[
 b(0,w)\leq\sup_{|\zeta|=r}b(0,\zeta)=\sup_K b
 \qquad(0<|w|\leq r).
\]
Consequently,
\[
 \{0\}^{n-1}\times\{0<|w|\leq r\}
 \subset\widehat K_{\mathcal B_c}.
\]
This set accumulates at the missing origin, so the hull is not
relatively compact in $\Omega$. Hence $\Omega$ is not
$\mathcal B_c$-convex.
\end{example}

For comparison, a domain $\Omega$ is \emph{hyperconvex} if there is
$\rho\in\PSH(\Omega)\cap C(\Omega)$ with $\rho<0$,
$\sup_\Omega\rho=0$, and
\[
 \{\rho<c\}\subset\subset\Omega\qquad\text{for every }c<0.
\]
Proposition~\ref{prop:Bc-examples} shows that every bounded hyperconvex
domain and every bounded pseudoconvex Runge domain is
$\mathcal B_c$-convex. We also write
$H^\infty(\Omega)=\mathcal O(\Omega)\cap L^\infty(\Omega)$, where $\mathcal O(\Omega)$ denotes the space of holomorphic functions
on $\Omega$.
If the hulls defined by the inequalities $|f(z)|\leq\sup_K|f|$
for all $f\in H^\infty(\Omega)$ are compact in $\Omega$, then
$\Omega$ is $\mathcal B_c$-convex as well, since
$|f|^2\in\mathcal B_c(\Omega)$ for every such $f$.

\begin{proposition}[Bounded approximation on $\mathcal B_c$-convex domains]
\label{prop:Bc-accessibility}
Let $\Omega\subset\C^n$ ($n\geq2$) be a bounded
$\mathcal B_c$-convex domain. For every
$\varphi\in\PSH(\Omega)\cap L^\infty(\Omega)$, there are bounded
continuous psh functions $b_j\to\varphi$ in $L^1_{\loc}(\Omega)$
that are maximal near $K_j$.  Moreover,
\[
 \mathscr R_{\mathrm{sm}}(\Omega,\varphi)=\cM_+(\Omega),
\]
and every vaguely closed subset of $\cM_+(\Omega)$, including the empty
set, is the full vague cluster set of the measures of one smooth bounded
psh sequence converging to $\varphi$ in $L^1_{\loc}(\Omega)$.  The conclusion of
Corollary~\ref{thm:moving-zero} also holds with smooth approximants.
\end{proposition}

\begin{proof}
By Lemma~\ref{lem:Bc-pseudoconvex}, $\Omega$ is pseudoconvex.
Lelong's approximation theorem \cite{Lelong} provides continuous maximal psh
functions $\psi_j$ on $\Omega$ such that
\[
 \|\psi_j-\varphi\|_{L^1(K_j)}<1/j.
\]
We construct bounded continuous psh functions $b_j$ that agree
with $\psi_j$ near $K_j$.

Fix $j$. By $\mathcal B_c$-convexity, choose an open set
$U_j\subset\subset\Omega$ containing
$\widehat{(K_j)}_{\mathcal B_c}$.
Lemma~\ref{lem:Bc-pseudoconvex} gives
$a_j\in\mathcal B_c(\Omega)$ satisfying
\[
 a_j\leq0\quad\text{on }K_j,\qquad
 a_j>1\quad\text{on }\partial U_j.
\]
Since $\psi_j$ is continuous, the numbers
$
 m_j=\min_{K_j}\psi_j,\
 M_j=\max_{\partial U_j}\psi_j
$
are finite. Choose
\[
 c_j<m_j-1,\qquad
 A_j>\max\{0,M_j-c_j+1\},
\]
and set $h_j=A_ja_j+c_j$. Then $h_j$ is globally bounded,
continuous, and psh. Moreover,
\[
 h_j\leq c_j<\psi_j-1\quad\text{on }K_j,
 \qquad
 h_j>A_j+c_j>\psi_j+1\quad\text{on }\partial U_j.
\]
By continuity, $h_j<\psi_j$ near $K_j$ and
$h_j>\psi_j$ near $\partial U_j$.

Define
\[
 b_j=
 \begin{cases}
  \max\{\psi_j,h_j\},&\text{on }U_j,\\
  h_j,&\text{on }\Omega\setminus U_j.
 \end{cases}
\]
Near $\partial U_j$, both expressions equal $h_j$.
Thus $b_j$ is continuous and psh on $\Omega$.
It is globally bounded because
\[
 \inf_\Omega h_j\leq b_j
 \leq
 \max\left\{\sup_{\overline U_j}\psi_j,\sup_\Omega h_j\right\}
 <\infty.
\]
Near $K_j$, we have $b_j=\psi_j$, so $b_j$ is maximal there and
\[
 \|b_j-\varphi\|_{L^1(K_j)}<1/j.
\]
Since the compact sets $K_j$ exhaust $\Omega$, it follows that
$b_j\to\varphi$ in $L^1_{\loc}(\Omega)$.

Furthermore, $\MA(b_j)$ vanishes near $K_j$.
Every compact subset of $\Omega$ is contained in $K_j$ for all
sufficiently large $j$, hence $\MA(b_j)\to0$ vaguely.
Lemma~\ref{lem:continuous-approximation} therefore gives
$
 0\in\mathscr R_{\mathrm{sm}}(\Omega,\varphi).
$
The remaining conclusions follow from Theorem~\ref{thm:structure}
and the smooth part of Corollary~\ref{thm:moving-zero}.
\end{proof}

\begin{proposition}[Examples of $\mathcal B_c$-convex domains]
\label{prop:Bc-examples}
Every bounded hyperconvex domain and every bounded pseudoconvex Runge
domain is $\mathcal B_c$-convex. There is a bounded balanced
$\mathcal B_c$-convex domain that is not hyperconvex.
By Chen's theorem \cite{Chen2021}*{Theorem~1.1}, every bounded
pseudoconvex domain with H\"older boundary is also $\mathcal B_c$-convex.
\end{proposition}

\begin{proof}
Suppose first that $\Omega$ is hyperconvex, and choose $\rho$ as in
its definition. Put $b=\max\{\rho,-1\}\in\mathcal B_c(\Omega)$.
For a nonempty compact $K\subset\subset\Omega$, set $s=\sup_K b<0$
and choose $c$ with $s<c<0$. Then
\[
 \widehat K_{\mathcal B_c}\subset\{b\leq s\}\subset\{\rho<c\}\subset\subset\Omega.
\]
This proves the hyperconvex assertion.

Suppose that $\Omega$ is bounded, pseudoconvex, and Runge.  Since
$|p|^2\in\mathcal B_c(\Omega)$ for every polynomial $p$, each
$z\in\widehat K_{\mathcal B_c}$ satisfies
$|p(z)|\leq\sup_K|p|$.  Polynomial approximation on $K\cup\{z\}$
gives the same inequality for every $f\in\mathcal O(\Omega)$.
Thus $\widehat K_{\mathcal B_c}$ lies in the holomorphic hull of $K$,
which is compact in $\Omega$ by holomorphic convexity of $\Omega$.

We construct a bounded $\mathcal B_c$-convex domain that is not
hyperconvex, following
\cite{JarnickiPflug2022}*{Example~2.12}.
Choose distinct points
$(a_j)_{j\geq1}\subset\partial\D\setminus\{1\}$ that are dense in
$\partial\D$, and define
\[
 s_j=\frac{2^{-j}}{1+|\log|1-a_j||},
 \qquad
 \alpha_j=\frac{s_j}{\sum_{k=1}^{\infty}s_k}.
\]
Then $\alpha_j>0$, $\sum_j\alpha_j=1$, and
\[
 \sum_{j=1}^{\infty}\alpha_j|\log|1-a_j||<\infty.
\]
Set
\[
 u(z_1,z_2)=\sum_{j=1}^{\infty}
              \alpha_j\log|z_1-a_jz_2|,
 \qquad
 h(z)=e^{u(z)}+\max\{|z_1|,|z_2|\},
\]
 and let
\[
 \Omega_*=\{z\in\C^2:h(z)<1\}.
\]
On any bounded polydisc centered at the origin, choose a constant
$C$ with $\log|z_1-a_jz_2|\leq C$ for every $j$.
The psh partial sums
$\sum_{j=1}^N\alpha_j(\log|z_1-a_jz_2|-C)$ decrease to $u-C$,
since $\sum_j\alpha_j=1$. Their limit is finite at points with
$z_2=0$ and $z_1\ne0$, so it is not identically $-\infty$.
Hence $u$ is psh on $\C^2$, and so is $h$.
Moreover, $h(\lambda z)=|\lambda|h(z)$ and
$h(z)\geq\max\{|z_1|,|z_2|\}$, so $\Omega_*$ is open, bounded,
balanced, and connected.
Moreover,
\[
 \inf\{t>0:z/t\in\Omega_*\}=h(z),
\]
so $h$ is its Minkowski functional.
Since its Minkowski functional $h$ is plurisubharmonic,
$\Omega_*$ is pseudoconvex by
\cite{JarnickiPflug2020}*{Proposition~2.2.31}.
We next prove that $\Omega_*$ is Runge.
Let $f\in\mathcal O(\Omega_*)$
and write its Taylor expansion at the origin as
$\sum_{\ell\geq0}P_\ell$, where $P_\ell$ is homogeneous of degree
$\ell$. For each compact $K\subset\subset\Omega_*$, choose $r>1$
such that $rK\subset\Omega_*$. By balance,
\[
 C:=\{\lambda z:z\in K,\ |\lambda|\leq r\}
 \subset\subset\Omega_*.
\]
Applying Cauchy's estimate to $\lambda\mapsto f(\lambda z)$ gives
\[
 \sup_K|P_\ell|\leq Mr^{-\ell},
 \qquad M=\sup_C|f|.
\]
Evaluating this Taylor series at $\lambda=1$ gives
\[
 \sup_K\left|f-\sum_{\ell=0}^NP_\ell\right|
 \leq\frac{Mr^{-N-1}}{1-r^{-1}}\longrightarrow0.
\]
Thus $\Omega_*$ is Runge and, being bounded and pseudoconvex,
is $\mathcal B_c$-convex by the preceding argument.
On the other hand, the choice of the weights implies
\[
 u(1,1)=\sum_{j=1}^{\infty}\alpha_j\log|1-a_j|>-\infty,
 \qquad h(1,1)=e^{u(1,1)}+1>1,
\]
whereas
\[
 u(a_j,1)=-\infty,\qquad h(a_j,1)=1.
\]
Choose a subsequence $a_{j_k}\to1$.
Then $(a_{j_k},1)\to(1,1)$ but
$ h(a_{j_k},1)=1<h(1,1),
$
so the Minkowski functional $h$ is discontinuous.
For bounded pseudoconvex balanced domains, hyperconvexity is
equivalent to continuity of the Minkowski functional; see
\cite{JPZ2000}*{Section~3, immediately before Proposition~3.8}.
Therefore $\Omega_*$ is not hyperconvex.

Finally, Chen's hyperconvexity theorem \cite{Chen2021}*{Theorem~1.1}
reduces the H\"older-boundary assertion to the hyperconvex case above.
\end{proof}

\subsection{Pluripolar deletion and completion of the classification}

Extension across a relatively closed pluripolar set gives another
source of bounded approximants.  We extend the given psh function,
approximate it on the larger domain with measures tending to zero,
and restrict the approximants.  Addition can then be carried out on
the punctured domain, even when it is not $\mathcal B_c$-convex.

\begin{corollary}[Removal of a pluripolar set]\label{cor:completion}
Let $G\subset\C^n$ ($n\geq2$) be a bounded domain, let $E\subset G$
be relatively closed and pluripolar, and put $\Omega:=G\setminus E$.
Every
$\varphi\in\PSH(\Omega)\cap L^\infty(\Omega)$ has a unique bounded psh
extension $\widetilde\varphi$ to $G$.  If $0\in\mathscr R_b(G,\widetilde\varphi)$, then
$0\in\mathscr R_b(\Omega,\varphi)$.  The same implication holds for
$\mathscr R_{\mathrm{sm}}$.  In particular, if $G$ is hyperconvex, then every bounded
psh function $\varphi$ on $G\setminus E$ satisfies
$\mathscr R_{\mathrm{sm}}(G\setminus E,\varphi)=\cM_+(G\setminus E)$
and admits exact realization of every vaguely closed cluster set,
including the empty set, by a smooth bounded psh approximating sequence.
\end{corollary}

\begin{proof}
The complement $\Omega$ is connected by
\cite{DemaillyCADG}*{Chapter~I, Corollary~5.26}.
The removable-singularity theorem
\cite{DemaillyCADG}*{Chapter~I, Theorem~5.24} gives a unique psh
extension $\widetilde\varphi$ to $G$, with
\[
 \widetilde\varphi(z)=\limsup_{\Omega\ni w\to z}\varphi(w)
 \qquad(z\in E).
\]
This formula preserves the global upper and lower bounds of
$\varphi$.

For $a\in\{b,\mathrm{sm}\}$, suppose
$0\in\mathscr R_a(G,\widetilde\varphi)$ and choose a realizing
sequence $v_j$ in that class. Its restrictions $u_j=v_j|_\Omega$
remain in the same class and converge to $\varphi$ locally in
$L^1$. Bedford--Taylor locality gives
$\MA(u_j)=\MA(v_j)|_\Omega$. For $g\in C_c(\Omega)$, its zero
extension $\widetilde g$ belongs to $C_c(G)$, since its support
is compact in $\Omega$. Thus
\[
 \int_\Omega g\,d\MA(u_j)
 =\int_G\widetilde g\,d\MA(v_j)\longrightarrow0,
\]
which proves $0\in\mathscr R_a(\Omega,\varphi)$ in both classes.
If $G$ is hyperconvex, Propositions~\ref{prop:Bc-examples}
and~\ref{prop:Bc-accessibility} give
$0\in\mathscr R_{\mathrm{sm}}(G,\widetilde\varphi)$.
The final assertions now follow from Theorem~\ref{thm:structure}.
\end{proof}

Only the psh function needs to extend across $E$.  A target Radon
measure on $G\setminus E$ may have infinite mass in every neighborhood
of $E$ and need not extend to a Radon measure on $G$.  We therefore
apply the addition theorem directly on $G\setminus E$.

\begin{proof}[Proof of Theorem~\ref{thm:classification}]
In case \textup{(1)}, the constant sequence $b_j=\varphi$ satisfies
\eqref{eq:zero-approximation}.  If $\varphi$ is continuous,
Lemma~\ref{lem:continuous-approximation} also gives
$0\in\mathscr R_{\mathrm{sm}}(\Omega,\varphi)$.
In cases \textup{(2)--(3)}, Propositions~\ref{prop:Bc-examples}
and~\ref{prop:Bc-accessibility} provide bounded continuous approximants
whose measures tend to zero.  Case \textup{(4)} follows from
Corollary~\ref{cor:completion}.

In each case, Theorem~\ref{thm:structure} gives
$\mathscr R_b(\Omega,\varphi)=\cM_+(\Omega)$ and realizes exactly the
vaguely closed subsets as full cluster sets, including the empty set.
Under the stated smoothness assumptions the same argument uses
$\mathscr R_{\mathrm{sm}}$.  Finally, for each individual $\nu$, the
equality of reachable sets supplies an entire sequence with measure
limit $\nu$ by definition.  Alternatively, take the constant target
$\mu_k=\nu$ in Corollary~\ref{thm:moving-zero}.
\end{proof}

\subsection{Examples beyond hyperconvex domains}

The following examples distinguish the domain conditions and illustrate
the effect of biholomorphic changes of coordinates and pluripolar deletion.
Write
\[
 \mathbb B^n=\{z\in\C^n:|z|<1\}.
\]

\begin{lemma}[Biholomorphic invariance]\label{lem:biholomorphic}
Let $\Phi:\Omega\to\Omega'$ be a biholomorphism between domains.
Composition with $\Phi$ preserves boundedness, continuity, smoothness, and
local $L^1$ convergence of psh functions.  For
$u\in\PSH(\Omega')\cap L^\infty_{\loc}(\Omega')$ and
$\nu\in\cM_+(\Omega')$,
\[
 \MA(u\circ\Phi)=\Phi^*\MA(u),
 \qquad \Phi^*\nu:=(\Phi^{-1})_*\nu.
\]
Consequently, hyperconvexity, the sets of possible limits, and condition
\eqref{eq:zero-approximation} are biholomorphically invariant.
\end{lemma}

\begin{proof}
Change of variables applies on compact sets because the real Jacobian of
$\Phi^{-1}$ is locally bounded.  The Monge--Amp\`ere identity follows first
for smooth potentials and then, locally on relatively compact coordinate
balls, from psh regularization and Bedford--Taylor monotone continuity.
Pull-back by the homeomorphism $\Phi$ preserves vague convergence.
\end{proof}

\begin{proposition}[$\mathcal B_c$-convexity is not necessary]
\label{prop:conditions-not-necessary}
Let $\Omega=\D^{n-1}\times\D^*$ ($n\geq2$).  Then
$0\in\mathscr R_{\mathrm{sm}}(\Omega,\varphi)$ holds for every
$\varphi\in\PSH(\Omega)\cap L^\infty(\Omega)$, although $\Omega$ is not
$\mathcal B_c$-convex and not hyperconvex.
\end{proposition}

\begin{proof}
Take $G=\D^n$ and $E=\D^{n-1}\times\{0\}$.  Since $G$ is hyperconvex,
Corollary~\ref{cor:completion} gives $0\in\mathscr R_{\mathrm{sm}}(\Omega,\varphi)$.
Example~\ref{ex:pseudoconvex-not-Bc} shows that $\Omega$ is not
$\mathcal B_c$-convex. Since every bounded hyperconvex domain is $\mathcal B_c$-convex by
Proposition~\ref{prop:Bc-examples}, this domain
is not hyperconvex.
\end{proof}

The punctured ball $\mathbb B^n\setminus\{0\}$, although not pseudoconvex,
is another instance of Corollary~\ref{cor:completion}.  In dimension two,
\[
 (z_1,z_2)\longmapsto(z_1/z_2,z_2)
\]
is a biholomorphism from the Hartogs triangle
$\{(z_1,z_2):|z_1|<|z_2|<1\}$ onto $\D\times\D^*$.  By
Lemma~\ref{lem:biholomorphic}, the Hartogs triangle is a bounded
pseudoconvex, nonhyperconvex domain for which every nonnegative Radon
measure occurs as a limit and every vaguely closed subset of
$\cM_+(\Omega)$ is realized as the full cluster set of a suitable sequence.

The bounded-domain construction applies locally on any pseudoconvex
domain.  Extending it to the whole domain preserves local bounds;
retaining global bounds requires suitable bounded pasting functions.

\section{General domains: localization and obstructions}\label{sec:localization}

A psh exhaustion allows a function constructed near a compact set to
be pasted to a function on the whole domain.  This gives locally
bounded approximants on arbitrary pseudoconvex domains.  Condition
\textup{(V)} below supplies bounded pasting functions and hence
globally bounded approximants.

\subsection{Locally bounded approximants}
Fix the test functions, exhaustion, and metric of
Lemma~\ref{lem:vague-metrization} for the domain under consideration.
The next lemma preserves a prescribed psh function near the supports
of the test functions, where Bedford--Taylor locality preserves its
measure.

\begin{lemma}[Extension by psh pasting]\label{lem:germ-localization}
Let $\Omega\subset\C^n$ be a domain, let $K\subset\subset\Omega$, and suppose that
$\rho\in\PSH(\Omega)\cap L^\infty_{\loc}(\Omega)$ and $d<c$ satisfy
\[
 K\subset\{\rho<d\},\qquad
 \overline{\{\rho<c\}}\subset G\subset\subset\Omega,
\]
where $G$ is a bounded domain.  If
$w\in\PSH(G)\cap L^\infty(G)$, then there is
$W\in\PSH(\Omega)\cap L^\infty_{\loc}(\Omega)$ that equals $w$ on
$\{\rho<d\}$.  If $\rho$ is globally bounded, then $W$ is globally bounded.
If both $\rho$ and $w$ are continuous, then $W$ can be chosen continuous.
\end{lemma}

\begin{proof}
Put $U=\{\rho<c\}$, $m_w=\inf_Gw$, and $M_w=\sup_Gw$.  Choose $A>0$ so that
\[
 A(c-d)>M_w-m_w+2,
\]
and set $B=m_w-1-Ad$ and $h=A\rho+B$.  On $\{\rho<d\}$ we have
$h<m_w-1\leq w-1$.  If $\xi\in\partial U$, upper semicontinuity gives
$\rho(\xi)\geq c$, and therefore
\[
 h(\xi)\geq m_w-1+A(c-d)>M_w+1.
\]
At every $\xi\in\partial U$,
\[
 \limsup_{U\ni z\to\xi}\max\{w(z),h(z)\}
 \leq\max\{M_w,h(\xi)\}=h(\xi).
\]
Consequently,
\begin{equation}\label{eq:germ-pasting}
 W=
 \begin{cases}
  \max\{w,h\},&\text{on }U,\\
  h,&\text{on }\Omega\setminus U
 \end{cases}
\end{equation}
is upper semicontinuous.  Away from $\partial U$ it is locally psh.
At a boundary point it equals $h$ and dominates $h$ everywhere, so it
satisfies the submean inequality on each complex line through that
point.  Thus $W$ is psh.  It is locally bounded and agrees with $w$ on
$\{\rho<d\}$.  If
$\rho$ is globally bounded, the two branches in \eqref{eq:germ-pasting}
have common global upper and lower bounds.  If $\rho$ and $w$ are continuous,
the strict boundary inequality also makes the pasting continuous.
\end{proof}

\begin{lemma}[Localization of the addition estimates]
\label{lem:local-addition}
Let $\Omega\subset\C^n$ ($n\geq2$) be a domain, and fix the test
functions, compact exhaustion, and metric of
Lemma~\ref{lem:vague-metrization}.
Suppose that for each $k$ there are data $\rho_k,d_k,c_k,G_k$ satisfying
the hypotheses of Lemma~\ref{lem:germ-localization} with $K=K_k$.
For any sequences
\[
 b_k\in\PSH(\Omega)\cap L^\infty_{\loc}(\Omega),\qquad
 \mu_k\in\cM_+(\Omega),
\]
there are locally bounded psh functions $u_k$ such that
\[
 \|u_k-b_k\|_{L^1(K_k)}\to0,\qquad
 d_{\mathrm v}\bigl(\MA(u_k),\MA(b_k)+\mu_k\bigr)\to0.
\]
If every $b_k$ and $\rho_k$ is continuous, the $u_k$ can be chosen
continuous.  If every $\rho_k$ is globally bounded, the $u_k$ can be
chosen globally bounded.  When both assumptions hold, these choices
can be made simultaneously.  Continuous $b_k$ and $\rho_k$ also allow
a smooth locally bounded choice of $u_k$.
\end{lemma}

\begin{proof}
The restriction $b_k|_{G_k}$ is bounded because $G_k\subset\subset\Omega$.
Apply Theorem~\ref{thm:fixed-addition} on $G_k$ and choose a bounded
psh function $w_k$ with
\[
 \|w_k-b_k\|_{L^1(K_k)}<1/k,\qquad
 \left|\int f_j\,d\bigl(\MA(w_k)-\MA(b_k)-\mu_k\bigr)\right|<1/k
 \quad(1\leq j\leq k).
\]
The measures here are restricted to $G_k$.
Lemma~\ref{lem:germ-localization} gives a locally bounded psh function
$W_k$ on $\Omega$ agreeing with $w_k$ on a neighborhood of $K_k$.
Bedford--Taylor locality preserves the estimates.  The first $k$
terms of \eqref{eq:vague-metric}, followed by its tail, give
\[
 d_{\mathrm v}\bigl(\MA(W_k),\MA(b_k)+\mu_k\bigr)
 \leq1/k+2^{-k}.
\]
The pasting preserves global boundedness when $\rho_k$ is globally
bounded.  If $b_k$ and $\rho_k$ are continuous, choose $w_k$ smooth in
Theorem~\ref{thm:fixed-addition}; then $W_k$ is continuous.

For the smooth choice, Richberg approximation gives smooth psh functions
$v_{k,\delta}$ with
\[
 W_k+\delta|z|^2\leq v_{k,\delta}
 \leq W_k+\delta(|z|^2+1).
\]
They converge locally uniformly to $W_k$ as $\delta\downarrow0$.
Bedford--Taylor continuity permits choosing $\delta_k$ so that the
$L^1(K_k)$ error and each of the first $k$ measure errors increase by
less than $1/k$.  Taking $u_k=v_{k,\delta_k}$ proves the smooth locally
bounded assertion, with metric error at most $2/k+2^{-k}$.
\end{proof}

\begin{proposition}[Approximation of a sum by locally bounded functions]
\label{prop:local-superposition}
Let $\Omega\subset\C^n$ ($n\geq2$) be a pseudoconvex domain.  Given
\[
 b_k\in\PSH(\Omega)\cap L^\infty_{\loc}(\Omega),
 \qquad \mu_k\in\cM_+(\Omega),
\]
there are $u_k\in\PSH(\Omega)\cap L^\infty_{\loc}(\Omega)$ such that
\[
 \|u_k-b_k\|_{L^1(K_k)}\longrightarrow0,\qquad
 d_{\mathrm v}\bigl(\MA(u_k),\MA(b_k)+\mu_k\bigr)\longrightarrow0.
\]
If every $b_k$ is continuous, every $u_k$ can be chosen smooth.
\end{proposition}

\begin{proof}
A pseudoconvex domain admits a smooth strictly psh exhaustion $\tau$
\cite{DemaillyCADG}*{Chapter~I, Theorem~7.2}.  For each $k$, choose
$d_k<c_k$ and a bounded domain $G_k\subset\subset\Omega$ such that
\[
 K_k\subset\{\tau<d_k\},\qquad
 \overline{\{\tau<c_k\}}\subset G_k\subset\subset\Omega.
\]
Apply Lemma~\ref{lem:local-addition} with $\rho_k=\tau$.
Its smooth assertion applies whenever the $b_k$ are continuous.
\end{proof}

For a locally bounded psh function $\varphi$, let
$\mathscr R_{\loc}(\Omega,\varphi)$ be defined as $\mathscr R_b$, with
locally bounded approximants in place of globally bounded ones.
Lelong's maximal approximants belong to this class and satisfy
$\MA(b_k)=0$.  The preceding proposition therefore gives the following
realization result.

\begin{corollary}[Prescribed measures with locally bounded approximants]
\label{cor:local-universality}
Let $\Omega\subset\C^n$ ($n\geq2$) be a pseudoconvex domain, and let
$\varphi\in\PSH(\Omega)\cap L^\infty_{\loc}(\Omega)$.  For every sequence
$(\mu_k)\subset\cM_+(\Omega)$ there are smooth locally bounded psh
functions $u_k$ on $\Omega$ such that
\[
 u_k\longrightarrow\varphi\quad\text{in }L^1_{\loc},\qquad
 d_{\mathrm v}(\MA(u_k),\mu_k)\longrightarrow0.
\]
Taking $\mu_k=\mu$ for any fixed $\mu\in\cM_+(\Omega)$ gives
$\mathscr R_{\loc}(\Omega,\varphi)=\cM_+(\Omega)$.  Global boundedness is not asserted.
\end{corollary}

\begin{proof}
Lelong's density theorem gives continuous maximal psh functions $b_k$ on
$\Omega$ converging to $\varphi$ locally in $L^1$; see \cite{Lelong} and
\cite{BacklundPersson}*{Theorem~2.1, p.~263}.  By
\cite{BacklundPersson}*{Theorem~2.3}, $\MA(b_k)=0$.  Apply the smooth part of
Proposition~\ref{prop:local-superposition} with target $\mu_k$.
\end{proof}

\subsection{Globally bounded approximants under a localization condition}

To retain global bounds, we need a bounded psh function for pasting in
place of the unbounded exhaustion.  We use the following condition:
\begin{equation*}\label{eq:condition-V}
 \tag{V}
 \forall K\subset\subset\Omega\quad
 \exists\,\rho_K\in\PSH(\Omega),\ \rho_K<0,\ \exists\,\alpha_K<0:
 \qquad K\subset\{\rho_K<\alpha_K\}\subset\subset\Omega.
\end{equation*}
Truncating $\rho_K$ from below preserves the indicated sublevel set and
makes the pasting function globally bounded.  The condition also supplies
pseudoconvexity, without a continuity assumption on $\rho_K$.

\begin{proposition}[Pseudoconvexity under \textup{(V)}]
\label{prop:V-pseudoconvex}
Every domain $\Omega\subset\C^n$ satisfying \textup{(V)} is pseudoconvex.
\end{proposition}

\begin{proof}
Choose a compact exhaustion $K_j\subset K_{j+1}^\circ$ and data
\[
 K_j\subset U_j:=\{\rho_j<\alpha_j\}\subset\subset\Omega
\]
from \textup{(V)}.  Replace $\rho_j$ by $\max\{\rho_j,-N_j\}$, with
$-N_j<\alpha_j$, to make it bounded without changing $U_j$.
Upper semicontinuity on $K_j$ gives
$m_j:=\max_{K_j}\rho_j<\alpha_j$.  The functions
\[
 \psi_j=\max\left\{\frac{\rho_j-m_j}{\alpha_j-m_j},0\right\}
\]
are nonnegative and psh, vanish on $K_j$, and are at least one on
$\Omega\setminus U_j$.
Every point has a neighborhood in some $K_J^\circ$.  On this
neighborhood, $\psi_j=0$ for all $j\geq J$, since $K_J\subset K_j$.
Thus $\tau=\sum_{j\geq1}\psi_j$ is a locally finite sum of finite-valued
psh functions, hence is finite-valued and psh.
For every positive integer $N$,
\[
 \{\tau<N\}\subset\bigcup_{j=1}^N U_j\subset\subset\Omega:
\]
outside this union, the first $N$ summands are all at least one.
Hence every sublevel of $\tau$ is relatively compact, and these
sublevels exhaust $\Omega$ because $\tau$ is finite everywhere.
The psh exhaustion criterion
\cite{DemaillyCADG}*{Chapter~I, Theorem~7.2(c)} now gives
pseudoconvexity; that criterion does not require the exhaustion to be
continuous.
\end{proof}

Since domains satisfying \textup{(V)} are Stein, Aytuna's theorem
identifies \textup{(V)} with Vogt's $\widetilde\Omega$ property for
$\mathcal O(\Omega)$ \cite{Aytuna2013}*{Theorem~1}.

\begin{proposition}[Approximation by globally bounded functions under \textup{(V)}]
\label{prop:V-superposition}
Let $\Omega\subset\C^n$ ($n\geq2$) be a domain satisfying condition~\textup{(V)}.
For arbitrary
\[
 b_k\in\PSH(\Omega)\cap L^\infty(\Omega),
 \qquad \mu_k\in\cM_+(\Omega),
\]
there are $u_k\in\PSH(\Omega)\cap L^\infty(\Omega)$ such that
\[
 \|u_k-b_k\|_{L^1(K_k)}\longrightarrow0,\qquad
 d_{\mathrm v}\bigl(\MA(u_k),\MA(b_k)+\mu_k\bigr)\longrightarrow0.
\]
The domain need not be bounded; its pseudoconvexity follows from
Proposition~\ref{prop:V-pseudoconvex}.
\end{proposition}

\begin{proof}
For each $k$, choose $\rho_k<0$ and $\alpha_k<0$ as in
\eqref{eq:condition-V}.  Replace $\rho_k$ by
$\max\{\rho_k,-N_k\}$, where $-N_k<\alpha_k$, to make it globally
bounded without changing $\{\rho_k<\alpha_k\}$.
Choose
\[
 \max_{K_k}\rho_k<d_k<c_k<\alpha_k
\]
and a bounded domain $G_k\subset\subset\Omega$ containing
$\overline{\{\rho_k<c_k\}}$.  Lemma~\ref{lem:local-addition}
applies with these globally bounded $\rho_k$ and gives the required
globally bounded functions.
\end{proof}

\begin{proposition}[Approximation with measures tending to zero under \textup{(V)}]
\label{prop:V-accessibility}
Let $\Omega\subset\C^n$ ($n\geq2$) be a domain satisfying condition~\textup{(V)}.  Then $0\in\mathscr R_b(\Omega,\varphi)$ for every
$\varphi\in\PSH(\Omega)\cap L^\infty(\Omega)$.
\end{proposition}

\begin{proof}
By Proposition~\ref{prop:V-pseudoconvex}, $\Omega$ is pseudoconvex.
Lelong's theorem therefore supplies continuous maximal psh functions
$\psi_k$ with $\|\psi_k-\varphi\|_{L^1(K_k)}<1/k$.  Choose bounded
localization data $\rho_k,d_k,c_k,G_k$ as in the proof of
Proposition~\ref{prop:V-superposition}.  The restriction
$\psi_k|_{G_k}$ is bounded, so Lemma~\ref{lem:germ-localization} produces a
globally bounded psh function $\beta_k$ equal to $\psi_k$ near $K_k$.
Thus $\beta_k\to\varphi$ locally in $L^1$, and
$\MA(\beta_k)\to0$ vaguely.
\end{proof}

\begin{corollary}[Question~2 under \textup{(V)}]
\label{cor:V-universality}
Let $\Omega\subset\C^n$ ($n\geq2$) be a domain satisfying condition~\textup{(V)}.  For every
$\varphi\in\PSH(\Omega)\cap L^\infty(\Omega)$ and every sequence
$(\mu_k)\subset\cM_+(\Omega)$ there are globally bounded psh functions $u_k$
such that
\[
 u_k\to\varphi\quad\text{in }L^1_{\loc},\qquad
 d_{\mathrm v}(\MA(u_k),\mu_k)\to0.
\]
Consequently, the equality
$\mathscr R_b(\Omega,\varphi)=\cM_+(\Omega)$ holds, and every vaguely
closed subset of $\cM_+(\Omega)$ is the exact vague cluster set of a bounded
psh approximating sequence; the empty set is allowed.
\end{corollary}

\begin{proof}
Let $b_k$ be the approximants from
Proposition~\ref{prop:V-accessibility}.
Proposition~\ref{prop:V-superposition}, applied to $(b_k)$ and $(\mu_k)$,
gives
\begin{align*}
 d_{\mathrm v}(\MA(u_k),\mu_k)
 &\leq d_{\mathrm v}\bigl(\MA(u_k),\MA(b_k)+\mu_k\bigr)
       +d_{\mathrm v}(\MA(b_k),0)\longrightarrow0.
\end{align*}
Taking $\mu_k=\mu$ for any fixed $\mu\in\cM_+(\Omega)$ gives
$\mathscr R_b(\Omega,\varphi)=\cM_+(\Omega)$.
Lemma~\ref{lem:cluster-topology} and the assertion for arbitrary measure sequences give every
nonempty closed cluster set; the sequence $k\delta_p$, for a fixed $p\in\Omega$, gives the empty set.
\end{proof}

Condition~\textup{(V)} is not necessary for the full cluster-set
classification.  The punctured domain
$\Omega=\D^{n-1}\times\D^*$ in
Proposition~\ref{prop:conditions-not-necessary} has that classification
but fails \textup{(V)}.  To check this, fix
$K=\{0\}^{n-1}\times\{|w|=r\}$, $0<r<1$.
Suppose that $\rho<0$ and $\alpha<0$ satisfied \eqref{eq:condition-V} for
this $K$.  If $h(w)=\rho(0,w)$ is identically $-\infty$, then
$\{\rho<\alpha\}$ contains the whole punctured slice.  Otherwise $h$ extends
subharmonically across zero
\cite{DemaillyCADG}*{Chapter~I, Theorem~5.23}.  Since
$\beta=\max_{|w|=r}h(w)<\alpha$, the maximum principle gives
$h(w)\leq\beta<\alpha$ for $0<|w|\leq r$.  In either case
$\{\rho<\alpha\}$ accumulates at the missing hypersurface and is not
relatively compact in $\Omega$.

\subsection{Obstructions on unbounded domains}

Global boundedness cannot be imposed on the locally bounded result for
every pseudoconvex domain.  On a cylinder with an entire
complex line as a factor, bounded psh functions are constant along that
factor.  Their top-degree Monge--Amp\`ere measures therefore vanish,
in contrast to the locally bounded conclusion of
Corollary~\ref{cor:local-universality}.

\begin{proposition}[Vanishing Monge--Amp\`ere measures on a cylinder]\label{prop:cylinder}
Let $\Omega=\C\times\D^{n-1}$ ($n\geq2$).  Then
\[
 \MA(u)=0\qquad
 \text{for every }u\in\PSH(\Omega)\cap L^\infty(\Omega).
\]
Consequently, for every bounded psh $\varphi$ on $\Omega$,
\[
 \mathscr R_b(\Omega,\varphi)=\{0\}.
\]
At $\varphi=0$, the constant zero sequence is smooth and maximal
on the whole domain.
\end{proposition}

\begin{proof}
Let $\pi:\C\times\D^{n-1}\to\D^{n-1}$ be the projection.  For fixed
$z'\in\D^{n-1}$, the function $\zeta\mapsto u(\zeta,z')$ is a
bounded-above subharmonic function on $\C$, hence constant.  Therefore
$u=\pi^*v$ for a bounded psh function $v$ on $\D^{n-1}$.

On each relatively compact ball in $\D^{n-1}$, choose smooth psh
regularizations $v_\ell\downarrow v$.  The Levi form of $\pi^*v_\ell$ has
rank at most $n-1$, so
$(\ddc\pi^*v_\ell)^n=0$.  Bedford--Taylor monotone continuity gives
$\MA(u)=0$.  The same argument applies to every globally bounded
approximating sequence, hence
$\mathscr R_b(\Omega,\varphi)=\{0\}$.  The constant sequence $b_k=\varphi$ has this limit.  At $\varphi=0$,
it is smooth and maximal on the whole domain.
\end{proof}

Smooth globally bounded approximation has a further obstruction.
Harz constructed an unbounded pseudoconvex domain carrying a bounded
negative psh function that is strictly
psh on a nonempty open set, while every continuously differentiable negative
psh function is constant there
\cite{Harz2023}*{Theorem~2 and Section~4}.  Every smooth globally bounded psh
function becomes negative after subtracting a constant and is therefore
constant on that open set.  On any ball compactly contained there, an
$L^1$ limit of such functions is almost everywhere constant, whereas Harz's
strictly psh function is not.  Thus the smooth approximants in
Corollary~\ref{cor:local-universality} cannot in general be required to be
globally bounded.

Thus locally bounded, globally bounded, and smooth globally bounded
approximation give different conclusions on unbounded domains.
Dimension and common bounds on a sequence impose further restrictions,
even on bounded domains.

\section{Restrictions and a further question}
\label{sec:dimension-one}

We record the restrictions in dimension one and under common local
bounds, then state the bounded approximation problem left open by
the preceding constructions.

\subsection{Dimension and common local bounds}

\begin{proposition}[Rigidity in complex dimension one]
\label{prop:dimension-one}
Let $\Omega\subset\C$ be a domain.  If subharmonic functions $u_k$ converge
to a subharmonic function $\varphi$ in $L^1_{\loc}(\Omega)$, then
\[
 \ddc u_k\longrightarrow\ddc\varphi
\]
vaguely as positive Radon measures.  Hence the cluster set is the singleton
$\{\ddc\varphi\}$.
\end{proposition}

\begin{proof}
Continuity of distributional differentiation gives
$dd^cu_k\to dd^c\varphi$ against every smooth compactly supported
test function.
For each compact $L\subset\subset\Omega$, choose
$\chi\in C_c^\infty(\Omega)$, $\chi\geq0$, equal to one near $L$.
Distributional convergence gives
$\sup_k\int\chi\,\ddc u_k<\infty$, hence a uniform mass bound on $L$.
The uniform approximation argument in Lemma~\ref{lem:vague-metrization}
therefore gives vague convergence.
\end{proof}

In every dimension, the Chern--Levine--Nirenberg estimate forces the
measures to converge to zero when the potentials tend to zero under
a common local bound.

\begin{proposition}[Convergence to zero under a common local bound]\label{prop:common-bound-zero}
Let $\Omega\subset\C^n$ be a domain.  Suppose that
$u_k\in\PSH(\Omega)\cap L^\infty_{\loc}(\Omega)$,
$u_k\to0$ in $L^1_{\loc}(\Omega)$, and
\[
 \sup_k\|u_k\|_{L^\infty(L)}<\infty
 \qquad\text{for every compact }L\subset\subset\Omega.
\]
Then $\MA(u_k)(K)\to0$ for every compact $K\subset\subset\Omega$.
In particular, $\MA(u_k)\to0$ vaguely.
\end{proposition}

\begin{proof}
Fix compact sets $K\subset L^\circ\subset L\subset P^\circ\subset
P\subset\subset\Omega$, and set $\beta=dd^c|z|^2$.  A nonnegative smooth cutoff
$\theta$ supported in $P^\circ$ and equal to one near $L$ gives
\[
 \int_L dd^cu_k\wedge\beta^{n-1}
 \leq\int_\Omega\theta\,dd^cu_k\wedge\beta^{n-1}
 =\int_\Omega u_k\,dd^c\theta\wedge\beta^{n-1}
 \leq C_{L,P}\|u_k\|_{L^1(P)}\longrightarrow0.
\]
In higher dimensions ($n\geq2$), Lemma~\ref{lem:CLN-mixed} with $b=h=u_k$ and $r=1$ gives
\[
 \MA(u_k)(K)
 \leq C_{K,L}\|u_k\|_{L^\infty(L)}^{n-1}
       \int_L dd^cu_k\wedge\beta^{n-1}\longrightarrow0.
\]
For $n=1$, the first estimate suffices.  Positivity then implies vague
convergence to zero.
\end{proof}

Local $L^1$ convergence to zero already gives locally uniform upper bounds
by the submean inequality for psh functions.  Consequently, any sequence
converging to zero with a nonzero Monge--Amp\`ere measure limit must
have no common lower bound on some compact subset.  The logarithmic truncations used here permit precisely
this behavior.

\subsection{The remaining bounded-domain question}

The boundedness step in maximal approximation remains unresolved on
an arbitrary bounded pseudoconvex domain.
\begin{question}\label{q:zero-accessibility}
Let $\Omega\subset\C^n$ ($n\geq2$) be a bounded pseudoconvex domain,
and let $\varphi\in\PSH(\Omega)\cap L^\infty(\Omega)$.  Do there
always exist bounded psh functions $b_j$
such that $b_j\to\varphi$ in $L^1_{\loc}$ and
$\MA(b_j)\to0$ vaguely?
\end{question}
Lelong's theorem supplies continuous maximal approximants on the whole
domain, but no global bound. $\mathcal B_c$-convexity,
condition~\textup{(V)} (which itself implies pseudoconvexity), and
restriction after pluripolar deletion provide the bounded approximants
used here.  If the maximal approximants are
already bounded above, lower truncation also suffices.

On every bounded domain, Theorem~\ref{thm:structure} shows that the
reachable set is closed and stable under addition of nonnegative Radon
measures, and characterizes the full cluster sets that can be realized.  A positive answer to
Question~\ref{q:zero-accessibility} would identify that reachable set
with $\cM_+(\Omega)$ on every bounded pseudoconvex domain, extending
the explicit classification proved here under the stated geometric
hypotheses.  This further approximation question is not identified
with the full content of Bedford's Question~2.

\appendix
\section{Divisorial terms in the general truncation limit}
\label{app:divisor}

The proof of the addition theorem combines $u=b+h$ with the Chern--Levine--Nirenberg
estimate and does not use this appendix.  We study the related
truncation in which a locally bounded psh function is placed inside
the maximum.  The resulting formula identifies the measure supported
on the divisor and every mixed term, including its power of
$\varepsilon$.

\begin{proposition}[General divisorial truncation limit]
\label{cor:general-divisor-limit}
Let $\Omega\subset\C^n$ ($n\geq2$) be a domain, and let
$D=\bigcup_{s=1}^N D_s$ be a reduced divisor with pairwise disjoint
smooth components, defined by $F\in\mathcal O(\Omega)$; thus $F$
vanishes to order one along each $D_s$.
Let $\Psi$ be locally bounded and psh, and let $V_0$ be smooth and psh.
Set
\[
 \sigma=[D]\wedge(\ddc V_0)^{n-1},\qquad
 q_\varepsilon=\Psi+\varepsilon\log|F|^2,\qquad
 v_\varepsilon=\varepsilon^{-1/(n-1)}V_0.
\]
For each fixed $\varepsilon>0$,
\begin{equation}\label{eq:program-limit}
 \MA\bigl(\max\{q_\varepsilon,v_\varepsilon-A\}\bigr)
 \longrightarrow R_\varepsilon\qquad(A\to+\infty),
\end{equation}
where
\begin{equation}\label{eq:defect-expansion}
 R_\varepsilon=\MA(\Psi)+\sigma
 +\sum_{p=2}^n\varepsilon^{(p-1)/(n-1)}
 [D]\wedge(\ddc\Psi)^{p-1}\wedge(\ddc V_0)^{n-p}.
\end{equation}
Each remainder is a positive Radon measure.  Consequently,
$R_\varepsilon\to\MA(\Psi)+\sigma$ vaguely as
$\varepsilon\downarrow0$.
\end{proposition}

\begin{proof}
Put $L=\log|F|^2$.  We first justify, for $0\leq j<n$, the identity
\begin{equation}\label{eq:divisor-BT-identity}
 \ddc\bigl(L(\ddc\Psi)^j\bigr)
 =[D]\wedge(\ddc\Psi)^j
 =\sum_{s=1}^N(i_s)_*\bigl((\ddc(\Psi|_{D_s}))^j\bigr),
\end{equation}
where $i_s:D_s\hookrightarrow\Omega$.  The product on the right is
the restriction product in \eqref{eq:divisor-restriction-product}.
For $j=0$ this is Poincar\'e--Lelong.  For $1\leq j<n$, the product
construction recalled before \cite{ABW2019}*{Theorem~2.1} gives local
integrability of the psh function $L$ against $(\ddc\Psi)^j$ and makes
$L(\ddc\Psi)^j$ a current with locally finite mass.

Work on a relatively compact coordinate ball and choose smooth psh
regularizations $\Psi_\nu\downarrow\Psi$ on a slightly smaller ball.
For smooth $\Psi_\nu$, Poincar\'e--Lelong and restriction of smooth
forms give
\begin{equation}\label{eq:smooth-divisor-BT}
 \ddc\bigl(L(\ddc\Psi_\nu)^j\bigr)
 =\sum_{s=1}^N(i_s)_*
       \bigl((\ddc(\Psi_\nu|_{D_s}))^j\bigr).
\end{equation}
Apply \cite{ABW2019}*{Theorem~2.1} with
\[
 u_\nu=u=L,\qquad
 v_{\ell,\nu}=\Psi_\nu,\qquad v_\ell=\Psi
 \quad(1\leq\ell\leq j).
\]
The multiplier sequence is constant, hence decreasing, and each other
factor decreases to a locally bounded psh function.  The theorem gives
$L(\ddc\Psi_\nu)^j\to L(\ddc\Psi)^j$ as currents, so their
distributional derivatives converge.  On each smooth $D_s$, the
restrictions decrease to $\Psi|_{D_s}$; Bedford--Taylor monotone
continuity gives convergence of their powers.  Push-forward is
continuous on compact supports.  Passing to the limit in
\eqref{eq:smooth-divisor-BT} proves \eqref{eq:divisor-BT-identity}.

For $q=q_\varepsilon$, locality off $D$ and the fact that both
non-pluripolar products and locally bounded Bedford--Taylor products
put no mass on $D$ give
\begin{equation}\label{eq:nonpolar}
 \langle\ddc q\rangle^p=(\ddc\Psi)^p\qquad(1\leq p\leq n).
\end{equation}
The integrability just proved shows that $q\in\mathcal G(\Omega)$:
$\langle\ddc q\rangle^j$ has locally finite mass and $q$ is locally
integrable against it for $0\leq j<n$.
Using the currents defined in the proof of Lemma~\ref{thm:awnw},
Bedford--Taylor recursion and \eqref{eq:divisor-BT-identity} yield
\begin{align}
 [\ddc q]^p
 &=\ddc\bigl(q(\ddc\Psi)^{p-1}\bigr)
   =(\ddc\Psi)^p+\varepsilon[D]\wedge(\ddc\Psi)^{p-1},
   \label{eq:bracket}\\
 S_p(q)&=\varepsilon[D]\wedge(\ddc\Psi)^{p-1}.
   \label{eq:singular-current}
\end{align}
Thus \cite{AWNW}*{Theorem~1.5} applies to $q$ and the smooth psh
function $v_\varepsilon$.  For $1\leq p<n$, its term
$S_p(q)\wedge(\ddc v_\varepsilon)^{n-p}$ has coefficient
\begin{equation}\label{eq:epsilon-balance}
 \varepsilon^{1-(n-p)/(n-1)}=\varepsilon^{(p-1)/(n-1)}.
\end{equation}
At $p=1$ the term is $\sigma$; the singular part of $[\ddc q]^n$
supplies the endpoint $p=n$.  This proves
\eqref{eq:program-limit}--\eqref{eq:defect-expansion}, with the unit
coefficients of the AWNW formula.
Each remainder is a finite sum of push-forwards of mixed
Bedford--Taylor products on the smooth components of $D$, so it has
locally finite mass.  The positive powers of $\varepsilon$ then give
the final vague limit, after taking $A\to\infty$ at fixed
$\varepsilon$.
\end{proof}

\section{Comparison in the vague metric}\label{app:metric}

Closeness in $d_{\mathrm v}$ need not imply convergence of signed
measure differences against all $C_c$ test functions. The example
below explains this distinction; the proposition shows how common
local mass bounds restore that implication.

\begin{example}[Dependence of measure comparison on the chosen metric]
\label{ex:metric-dependence}
Choose $p\in\Omega$, a unit vector $e$, and $r_k\downarrow0$ so that the
segments from $p$ to $p+r_ke$ lie in $\Omega$.  Put
\[
 M_k=r_k^{-1/2},\qquad
 \alpha_k=M_k\delta_{p+r_ke},\qquad
 \beta_k=M_k\delta_p.
\]
For each fixed $f\in C_c^\infty(\Omega)$, the mean-value estimate on a fixed
neighborhood of $p$ gives
\[
 \left|\int f\,d(\alpha_k-\beta_k)\right|
 =M_k|f(p+r_ke)-f(p)|
 \leq C_fM_kr_k=C_fr_k^{1/2}\longrightarrow0.
\]
Dominated convergence in the series \eqref{eq:vague-metric} therefore yields
$d_{\mathrm v}(\alpha_k,\beta_k)\to0$ for the fixed smooth test family.

Choose $\chi\in C_c^\infty(\Omega)$ equal to one on a neighborhood of $p$ and of all
$p+r_ke$ for large $k$, and set
\[
 g(z)=\chi(z)|z-p|^{1/2}.
\]
Then, for all sufficiently large $k$,
\[
 \int g\,d(\alpha_k-\beta_k)=M_kr_k^{1/2}=1.
\]
Thus convergence fails for this $C_c$ test function. Adjoining $g$
as the first test function gives another metric for the vague
topology, with distance at least $1/4$ for large $k$. Closeness of
moving sequences therefore depends on the compatible uniform
structure. This second metric contains a nonsmooth continuous test
function; the example does not compare two purely smooth test families.
\end{example}

\begin{proposition}[Convergence of signed differences under local mass bounds]
\label{prop:mass-control}
Let $\Omega\subset\C^n$ be a domain, and let $(\alpha_k)$ and
$(\beta_k)$ be sequences in $\cM_+(\Omega)$.  Assume
that, for every compact $L\subset\subset\Omega$,
\[
 \sup_k\alpha_k(L)<\infty,
 \qquad
 \sup_k\beta_k(L)<\infty,
\]
and that $d_{\mathrm v}(\alpha_k,\beta_k)\to0$.  Then
\[
 \int_\Omega g\,d(\alpha_k-\beta_k)\longrightarrow0
 \qquad\text{for every }g\in C_c(\Omega).
\]
\end{proposition}

\begin{proof}
For real $g\in C_c(\Omega)$, choose $m$ with $\supp g\subset U_m$
and put $L=L_m$, using Lemma~\ref{lem:vague-metrization}.
Given $\delta>0$, choose $f\in\mathcal D_m$ with
$\|g-f\|_\infty<\delta$. Both supports lie in $L$, so
\[
 \left|\int g\,d(\alpha_k-\beta_k)\right|
 \leq\left|\int f\,d(\alpha_k-\beta_k)\right|
   +\delta\bigl(\alpha_k(L)+\beta_k(L)\bigr).
\]
The first term tends to zero by the metric hypothesis and the
masses are uniformly bounded. Let $k\to\infty$, then
$\delta\downarrow0$; complex $g$ follows by linearity.
\end{proof}

\section*{Acknowledgments}

The author used ChatGPT (OpenAI) during the preparation of the manuscript
for exploratory drafting and proof-audit assistance.  All mathematical arguments,
references, and conclusions were independently verified by the author, who
assumes full responsibility for the contents of the paper.

\end{document}